\documentclass[aap,preprint]{imsart}

\RequirePackage{amsthm,amsmath,amsfonts,amssymb}
\RequirePackage[numbers]{natbib}
\RequirePackage[colorlinks,citecolor=blue,urlcolor=blue,bookmarksdepth=subsubsection]{hyperref}%% uncomment this for coloring bibliography citations and linked URLs
\RequirePackage{graphicx}%% uncomment this for including figures

\usepackage[utf8]{inputenc}
\usepackage[dvipsnames]{xcolor}
\usepackage{colortbl,mathrsfs,mathtools,dsfont,appendix,float,lineno,ulem}
\usepackage[shortlabels]{enumitem}
\usepackage[nameinlink]{cleveref}
\usepackage{subcaption}
\usepackage{multirow}
\usepackage{stmaryrd}

\makeatletter
\def\namedlabel#1#2{\begingroup
    #2%
    \def\@currentlabel{#2}%
    \phantomsection\label{#1}\endgroup
}
\makeatother

\theoremstyle{plain}
\newtheorem{definition}{Definition}
\newtheorem{proposition}[definition]{Proposition}
\newtheorem{lemma}[definition]{Lemma}

\newtheorem{theorem}{Theorem}
\newtheorem{remark}[definition]{Remark}

\theoremstyle{definition}

\numberwithin{definition}{section}
\numberwithin{equation}{section} %Equations are numbered as #section.#equation

\DeclareMathOperator{\diam}{diam}

\DeclareMathOperator{\supp}{supp}
\DeclareMathOperator{\Var}{Var}
\DeclareMathOperator{\Cov}{Cov}
\renewcommand{\Pr}{\mathrm P}
\newcommand{\grad}{\nabla}
\renewcommand*{\P}{\mathbb{P}}
\newcommand*{\E}{\mathbb{E}}
\newcommand*{\Eu}{\mathbb{E}_{u}}
\newcommand*{\R}{\mathbb{R}}

\newcommand*{\Z}{\mathbb{Z}}

\newcommand*{\N}{\mathbb{N}}

\newcommand*{\Qcal}{\mathcal{Q}}

\newcommand*{\Fcal}{\mathcal{F}}

\newcommand*{\Gcal}{\mathcal{G}}
\newcommand*{\Ccal}{\mathcal{C}}

\newcommand*{\Dcal}{\mathcal{D}}

\newcommand*{\Lcal}{\mathcal{L}}

\newcommand*{\Rcal}{\mathcal{R}}

\newcommand*{\bx}{\mathbf{x}}
\newcommand*{\bff}{\mathbf{f}}
\newcommand*{\bX}{\mathbf{X}}
\newcommand*{\bG}{\mathbf{\Gamma}}
\newcommand*{\bg}{\boldsymbol{\gamma}}
\newcommand{\norm}[1]{\left \lVert  #1 \right \rVert}

\newcommand*{\1}{\mathds{1}}

\renewcommand*{\L}{\Lambda}

\renewcommand*{\d}{\mathrm{d}}
\newcommand*{\e}{\mathrm{e}}

\renewcommand{\e}{\mathrm{e}}

\crefname{equation}{}{}

\definecolor{amethyst}{rgb}{0.6, 0.4, 0.8}

\begin{document}
\begin{frontmatter}

\title{Existence, properties, and parametric inference for possibly hyperuniform Gibbs perturbed lattices}
\runtitle{Gibbs perturbed lattices}

\begin{aug}
%%%%%%%%%%%%%%%%%%%%%%%%%%%%%%%%%%%%%%%%%%%%%%%
%% Only one address is permitted per author. %%
%% Only division, organization and e-mail is %%
%% included in the address.                  %%
%% Additional information can be included in %%
\author[A]{\fnms{Jean-François}~\snm{Coeurjolly} \ead[label=e1]{jean-francois.coeurjolly@univ-grenoble-alpes.fr}}
\and
\author[A]{\fnms{Christopher}~\snm{Renaud Chan}\ead[label=e2]{christopher.renaud-chan@univ-grenoble-alpes.fr}}
%\thanks{[\textbf{Corresponding author}]}

\address[A]{Univ. Grenoble Alpes, CNRS, LJK, 38000 Grenoble, France \printead[presep={,\ }]{e1,e2}}

\end{aug}

%% Abstract
\begin{abstract}
This work lies at the intersection of Gibbs models and hyperuniform point processes. Classical Gibbs models, whether defined on lattices or in continuous space, provide flexible tools to describe interacting particle systems but are generally not hyperuniform. Conversely, known hyperuniform models such as the Ginibre process or perturbed lattices lack flexibility and typically cannot enforce physically relevant constraints such as hard-core interactions. We introduce a new class of models, termed Gibbs perturbed lattice models, which preserve a lattice structure while allowing interactions through a Hamiltonian defined on the perturbed particle locations. We establish existence results for the associated Gibbs measures, derive DLR-type equilibrium equations, and show that some models in this class exhibit hyperuniformity. Finally, we propose statistical inference methods based on the Takacs-Fiksel type approach and prove their asymptotic properties.
\end{abstract}





\end{frontmatter}
	
\section{Introduction} ${ }$\\

\noindent {\bf Gibbs models.} Gibbs processes form a particular class of stochastic models designed to describe interacting geometric objects (e.g., points). Such processes are typically defined either on a lattice, say $S=\mathbb Z^d$, or on a continuous space, say $S=\mathbb R^d$. General references include \citep{gaetan2010spatial,georgii1979canonical,georgii2011book,ruelle1970superstable}. 
More precisely, when $S=\mathbb Z^d$, one speaks of a \textit{Gibbs lattice system}, denoted by $\bX=\{X_i,  i\in\mathbb Z^d\}$. The locations $i$ are fixed, and the model takes values in $E^{\mathbb Z^d}$, where $E$ denotes the state space. The model is specified through conditional distributions and a Hamiltonian depending on the variables $X_i$. Well-known examples include the Ising model, where $E=\{-1,1\}$, the Potts model, where $E=\{a_0,\dots,a_{K-1}\}$ with $K\ge2$ (see, e.g., \citep{friedli2017statistical}), and Besag autoregressive-type models \citep{besag1974spatial}. When $S=\mathbb R^d$, one speaks of a \textit{Gibbs point process}, denoted by $\bX=\{x_1,\dots,x_n,\dots\}$ with $x_i\in\mathbb R^d$. For standard Gibbs point processes, both the locations and the number of points are random, and the model can again be described through conditional densities with respect to the Poisson measure, see, for example, \citep{dereudre2019introduction}. Interactions between points are encoded via an energy function, or Hamiltonian. When this Hamiltonian depends on pairwise distances, one obtains pairwise interaction point processes, including the classical Strauss model \citep{strauss1975model} and the Lennard-Jones model (see, e.g., \citep{ruelle1970superstable}). Models with higher-order interactions include, for instance, the Quermass model \citep{kendall1999quermass}, which encompasses the area-interaction and Widom-Rowlinson models as special cases, and the Delaunay-Voronoi interaction models \citep{dereudre2011practical}. Gibbs lattice systems and Gibbs point processes constitute two fundamentally different classes of models, each with its own characteristics, mathematical tools, and theoretical results. In particular, the existence and phase transition analysis for infinite-volume Gibbs point processes relies on techniques that differ from those used in the lattice setting \citep{dereudre2019introduction,dereudre2012existence,georgii1979canonical}. \\

\noindent{\bf Hyperuniformity.} Hyperuniformity is a relatively recent concept in probability theory. Introduced in the physics literature, most notably by \citep{torquato2003local} (see also the surveys by \citep{torquato2016hyperuniformity,torquato2018hyperuniform}). It aims to describe systems of rigid particles exhibiting an unusual form of large-scale order. Roughly speaking, a point process is said to be hyperuniform if $\Var[N(B(0,r))]/|B(0,r)|$, that is the normalized variance of the number of points falling in a Euclidean ball with radius $r$, tends to 0 as $r\to \infty$. This condition expresses that the fluctuations of the counting variable grow more slowly than those of a homogeneous Poisson point process. We refer the reader to \citep{lachièzerey2025hyperuniformrandommeasurestransport} for a recent and comprehensive probabilistic and statistical overview of hyperuniform point processes. Hyperuniformity is not simply related to repulsiveness. For instance, Matérn hard-core point processes are known not to be hyperuniform (see, e.g., \citep{KiderlenHoerig2013}). Rather, hyperuniformity is a global and subtle property.

Several classes of point processes are known to be hyperuniform, including the Sine-$\beta$ process \citep{valko2009continuum}, the Ginibre process, as particular examples of (continuous) determinantal point processes (see, e.g., \citep{lavancier2015determinantal}), and the zero set of Gaussian Analytic Functions \citep{hough2009zeros}. Another important class consists of perturbed lattices, defined as $\bG=\{i+X_i, i\in\Lcal\}$, where $\Lcal$ is a lattice and $(X_i)_i$ is a random field indexed by~$\Lcal$. When the random variables $X_i$, called \textit{moves}, are i.i.d. with common distribution $\mathbb Q$, \citep{dereudre2024non2} provide conditions on $\mathbb Q$ and on the dimension $d$ that ensure hyperuniformity of~$\bG$. Although the Ginibre and perturbed lattice models are simple to define and interpret, they suffer from important limitations. In particular, it is not possible to impose a hard-core condition of the form $|i+X_i-j-X_j|\ge\delta>0$ for $i\neq j$. Such a constraint is, however, of significant interest in physics \citep{buisine2025absorption}. One might conjecture that this issue could be addressed through a dependent subsampling of a hyperuniform point process; however, such a procedure destroys the hyperuniformity property \citep{lachièzerey2025hyperuniformrandommeasurestransport}.\\

\noindent{\bf Gibbs perturbed lattice  model}. At this point, one may wonder whether continuous Gibbs point processes, which are natural and flexible models able to produce repulsive point patterns, can be hyperuniform. The answer is in fact negative, as shown recently by \citep{dereudre2024non}. The authors prove that most Gibbs point processes defined on $\mathbb R^d$ are not hyperuniform and fail to identify any hyperuniform example within this class. In this paper, we introduce a new class of models lying at the intersection of perturbed lattices and continuous Gibbs point processes. On a finite sub-lattice, $\Delta \subset \Lcal$, and for ease of exposition, a Gibbs perturbed lattice model is defined through a density with respect to $\mathbb Q^{\otimes\Delta}$. It is worth emphasizing that our framework fundamentally differs from both Gibbs lattice systems and continuous Gibbs point processes. Indeed, the resulting point process remains of the form $\bG=\{i+X_i, i\in\Lcal\}$. Each lattice site corresponds to a unique point of $\bG$, so that $\bG$ can also be interpreted as a lattice random field. The key difference, however, is that the Hamiltonian does not depend on the values of the random field itself, but instead on the locations of the perturbed points, that is, on the elements of $\bG$.\\

\noindent{\bf Contributions.} Given the specific structure of our model, important theoretical questions, such as the existence of infinite-volume Gibbs measures and the formulation of Dobrushin-Lanford-Ruelle (DLR) equations \citep{dereudre2019introduction}, cannot be directly addressed using existing results for classical Gibbs lattice systems or continuous Gibbs point processes (see, e.g., \citep{georgii1979canonical,georgii2011book,ruelle1970superstable}). We tackle these issues in the present paper. We provide sufficient conditions on the move distribution $\mathbb Q$ and on the Hamiltonian of the particle system that ensure the existence of stationary Gibbs perturbed lattice models, both in finite and infinite volumes. Building on a general result obtained by~\citep{dereudre2024non2}, we then show that some models within this class are hyperuniform. In addition, we introduce a broad family of equilibrium equations (see Section~\ref{sec:dlr}), referred to as first-order, second-order, and variational DLR equations, which can be viewed as analogues of the corresponding first-order, second-order, and variational Georgii-Nguyen-Zessin (GNZ) equations \citep{baddeley2013variational,georgii1979canonical,nguyen1977integral}.

Having at our disposal a broad new class of parametric models, we then address the problem of estimating a parameter $\theta\in\mathbb R^p$ ($p\ge1$) from a single realization of a Gibbs perturbed lattice model. We consider two observation settings, summarized in Figure~\ref{fig:framework}: either the pairs $(i,i+X_i)$ are observed, or only the perturbed locations $(i+X_i)$. Due to the presence of an intractable normalizing constant in the model definition, maximum likelihood estimation is impractical, and likelihood-based approximations are computationally expensive, as is also the case for standard Gibbs lattice systems and continuous Gibbs point processes, see e.g.~\citep{gaetan2010spatial}. We therefore propose alternative inference methods based on pseudo-likelihood \citep{baddeley2000practical,besag1974spatial,jensen1991pseudolikelihood} and variational approaches (see, e.g., \citep{almeida1993variational,baddeley2013variational}). We show that both approaches can be formulated as particular cases of the Takacs–Fiksel method, originally introduced for continuous Gibbs point processes via Georgii–Nguyen–Zessin (GNZ) equations \citep{coeurjolly2012takacs,fiksel1984estimation,fiksel1988estimation,takacs1986estimator}. In our framework, the Takacs–Fiksel estimator is defined by minimizing a contrast function based on the $L^2$ distance between empirical versions of DLR equations evaluated on $L\ge p$ test functions. It is worth emphasizing that the pseudo-likelihood and Takacs–Fiksel contrast functions differ substantially from their counterparts for Gibbs lattice systems and continuous Gibbs point processes, as they require nonstandard and careful border corrections involving both $i$ and $i+x$ for $x\in\mathbb R^d$. We work within an increasing-domain asymptotic framework and establish, using in particular novel ergodic results for empirical versions of DLR equations, that the resulting Takacs–Fiksel estimator of $\theta$ is consistent. Moreover, we derive its rate of convergence.\\

\noindent{\bf Outline.} The remainder of the paper is organized as follows. Section~\ref{sec:background} reviews Gibbs lattice systems and Gibbs point processes and introduces the main assumptions imposed on the move distribution $\mathbb Q$ and on the Hamiltonian of the particle system considered throughout the paper. Section~\ref{sec:existence} is devoted to the definition of a Gibbs perturbed lattice  models, the existence of the associated Gibbs measures, their main properties, and the derivation of DLR-type equations. Section~\ref{sec:inference} addresses parametric statistical inference for the Gibbs perturbed lattice model. We present several inference methodologies and establish asymptotic results within the increasing-domain framework. Section~\ref{sec:simulation} provides illustrative examples of the proposed model and statistical methodology. Finally, the proofs of the main results are deferred to~\ref{app:aux.existence}–\ref{app:convTF}.

\section{Background and notation} \label{sec:background}
	
\subsection{Lattice, perturbed lattices and point processes}
	
Let $\Lcal$ be a full rank lattice, subgroup of $\R^d$ generated by linear combinations with integer coefficients of the vectors $(a_1, \dots, a_d)$ that forms a basis of $\R^d$. The fundamental domain associated to a lattice $\mathcal{L}$  is the set $$ \mathcal{D} := \left\{ \sum_{i=1}^d t_i a_i, (t_1, \dots, t_d) \in [0,1)^d\right\}.$$ The fundamental domain is the polyhedra that is used to pave $\R^d$ according to the lattice. We denote by $\delta := \min\{ |i|, i \in \Lcal\}$ the size parameter of the lattice. The basic example is $\Lcal=\Z^d$ and $\Dcal=[0,1)^d$. In this paper, the lattice $\Lcal$ is viewed as the origin of the particles. We consider $\R^d$ to be the mark space, also called spin space or state space, which describes the displacement of a single particle from their lattice positions. We denote by $\mathbb{Q}$ a probability measure on $\R^d$, which serves as the reference distribution for this single displacement  particle. For each $i \in \Lcal$, we denote by $x_i \in \R^d$ the displacement (or perturbation) of the particle at site $i$.  The lattice system configuration space is denoted by $\Omega := (\R^d)^\Lcal$, and is endowed with the product $\sigma$-algebra, and a configuration is denoted by $\bx := (x_i)_{i\in \Lcal}$.  For any subset $\Lambda \subset \Lcal$, the configuration on the sub-lattice is denoted by $\Omega_\Lambda := (\R^d)^\Lambda$, and the projection of a configuration $\bx \in \Omega$ into $\Omega_\Lambda$ is denoted by $\bx_\Lambda := (x_i)_{i\in \Lambda}$. A random field on the lattice $\Lcal$, denoted by $\bX$, is therefore a random variable on $\Omega$. In the following, we let $U\sim \mathcal U_\Dcal$ be a random variable uniformly distributed on $\Dcal$. For $B\subseteq \R^d$, we let $\1_B: \R^d \to \{0,1\}$ denote the indicator function on $B$. Finally, for any vector $z\in \R^k$, $k\ge 1$ or any bounded set $B\subset \R^d$, the notation $|z|$ and $|B|=\lambda^d(B)$ respectively stand for the Euclidean norm of $z$ and the volume (the $d$-dimensional Lebesgue measure $\lambda^d$) of $B$.

A perturbed lattice, denoted by $\bG$, is a random variable from $\Omega \times \R^d$ to the set of all countable subsets of $\R^d$ such that $\bG(\bX, U) := \{ i + X_i + U, i\in \Lcal \}$, where we require that particles arising from a perturbed lattice are indistinguishable. Furthermore, for any subset $\Lambda \subset \Lcal$, the perturbed sub-lattice $\bG_\Lambda$ is the random variable $\bG_\Lambda(\bX_\Lambda, U) := \{ i + X_i + U, i \in \Lambda \}$. We denote by $\Lcal_U := \{ i+U, i \in \Lcal \}$ the shifted lattice. This shifted lattice is by construction invariant under any translation on $\R^d$. Therefore, $U$ stationarizes the underlying lattice from which we build the perturbed lattice. From now on, unless explicitly specified and when there is no ambiguity, the notation $\Pr,\E[\cdot],\Var[\cdot],\Cov[\cdot,\cdot]$, standing for probability, expectation, variance and covariance, is used without indicating the underlying probability measure(s). A point configuration in $\R^d$ is in general denoted by $\bg$ and  is a locally finite set in $\R^d$, i.e. for any bounded Borel set $W$ of $\R^d$, $\# \bg \cap W$ is finite. The set of point configurations is denoted by $\mathcal{C}$ and the set of finite point configurations is denoted by $\mathcal{C}_f$. We endow $\Ccal$ with the $\sigma$-algebra generated by the counting function $N_W$ for all Borel sets $W$ in $\R^d$. In the following, for any $W \subset \R^d$  and $\bg \in \Ccal$ we denote by $\bg_W := \bg \cap W$, the restriction of $\bg $ to $W$. A point process is a random variable on $\Ccal$.  For $\bg \in \Ccal$ and $v \in \R^d$, we denote by $\bg + v$ the translation of all points of $\bg$ by $v$. A point process $\bG$ is said to be stationary when its distribution is translation invariant, i.e. the distributions of $\bG + v$ and $\bG$ are identical for any $v \in \R^d$. 

\subsection{Hamiltonian of a system and move distribution} \label{sec:hamiltonian}

The Gibbs perturbed lattice model depends mainly on two ingredients: the Hamiltonian $H$ which measures the interaction between particles of the system and the mark (or move) distribution $\mathbb Q$. In this section, we provide definitions, examples and the main assumptions on $H$ and $\mathbb Q$.

\paragraph{Hamiltonian or energy functional}

We consider models of interaction that are common in the canonical and grand canonical Gibbs point process setting. The Hamiltonian, also called the energy functional, is a measurable function, $H : \mathcal{C}_f \to \R \cup \{\infty\}$. Let $W \subset \R^d$ be a bounded domain, the local energy on $W$ of a configuration $\zeta \in \Ccal$ is defined as 
\begin{equation}
    H_W(\bg) := \lim\limits_{n \to \infty} \left( H(\bg_{[-n,n]^d}) - H(\bg_{[-n,n]^d \setminus W})   \right)
\end{equation}
if the limit exists, where we use the convention  $\infty - \infty =0$. Another quantity of interest is the local energy of a point which measures the contribution of a particle at a certain position to the energy of the whole configuration. Let $\zeta \in \Ccal$ and $x \in \R^d$,  the local energy is a measurable function from $\R^d \times \Ccal$ to $\R \cup \{\infty\}$ given by
\begin{equation}
h(x,\bg) := \lim\limits_{n \to \infty} \left( H(\{x\} \cup \bg_{[-n,n]^d}) - H(\bg_{[-n,n]^d})  \right).
\end{equation}
We impose a few definitions, restrictions on $H$. These assumptions are in particular required to ensure the existence of a Gibbs perturbed lattice in the infinite volume, see Theorem~\ref{thm:existence}.
\begin{enumerate}[($\mathcal H$1)]
\item $H$ is stable, i.e. there exists $A \geq 0$ such that for all $\bg \in \mathcal{C}_f$, \ $H(\bg) \geq  - A N(\bg)$.  \label{H:stability}
%{\sc [Stability]}. There exists $A \geq 0$ such that for all $\bg \in \mathcal{C}_f$, \ $H(\bg) \geq  - A N(\bg)$.  \label{H:stability}
\item $H$ is non-degenerate, i.e. there exists $\varepsilon >0$ and $B\geq0$ such that %\JF{$\mathbb Q$ même pas défini}\sout{: (i) $\mathbb{Q}(B(0,\varepsilon))>0$; (ii)} 
for all $i \in \mathcal{L}$, $|x_i| \leq \varepsilon$, and for any finite $\Lambda \subset \mathcal{L}$, $H(\{i+x_i\}_{i\in\Lambda}) \leq B |\Lambda|$.  \label{H:nondegenerate}
\item $H$ is hereditary, i.e. for any $\bg \in \mathcal{C}_f$ such that  $H(\bg) = \infty$, then, for any $x \in \R^d$,  $H(\bg \cup \{x\}) =  \infty$. \label{H:heredity}
\item $H$ is invariant under translation, i.e. for all $v \in\R^d$ and $\bg \in \mathcal{C}_f$, $H(\bg) = H(\bg + v)$. \label{H:invariant}
\end{enumerate}        
\begin{enumerate}[($\mathcal H5.$1)]
\item $H$ has a finite range, i.e. there exists $R>0$ such that for any $W \subset \R^d$ bounded, 
%the local energy on $W$ depends only on points of $\bg_{W \oplus B(0,R)}$, i.e. 
$H_W(\bg) =  H_W (\bg_{W \oplus B(0,R)})$,
where $W \oplus B(0,R) := \{ x + y, x \in W, y \in B(0,R) \}$. \label{H:range}
\item %{\sc Summable pairwise interaction}. 
There exists $\phi : \R^+ \to \R \cup \{\infty\}$ monotone at infinity and $\epsilon>0$ such that
\begin{equation}
H(\bg) = \sum_{\{x,y\} \subset \bg} \phi(|x-y|), \quad \text{and } \quad \int_{\R^+ \setminus  [0,\epsilon]} r^{d-1} \phi(r) \d r< \infty. 
\end{equation} \label{H:summable}
\item %{\bf Bounded local energy}. 
The local energy function $h$ is uniformly bounded. \label{H:bounded}
\end{enumerate}

Let us discuss these assumption through the presentation of a few examples of standard Hamiltonians used in the literature \citep{kendall1999quermass,moller2003statistical,ruelle1970superstable}. The first three ones are pairwise interaction point processes, while the last one acts on geometric features of a random set. All these examples satisfy (at least) \ref{H:stability}-\ref{H:invariant} (with some restriction on the parameter space for some of them). 
\begin{enumerate}
    \item {\it Strauss interaction with possible hard-core interaction}. Let $A \in \R$ and $0\le r<R$.  For any $\bg \in \Ccal_f$, the Hamiltonian is given by
        \begin{equation}
        H(\bg) = \sum_{\{x,y\} \in \bg} A \1_{[r,R]}(|x-y|) + (\infty) \1_{[0,r]}(|x-y|)
        \end{equation}
    This model satisfies \ref{H:range} (and thus \ref{H:summable}). Note in particular that \ref{H:stability} is satisfied if $r=0$ and $A \geq 0$ or if $r>0$ and that \ref{H:nondegenerate} is satisfied for $\varepsilon < \delta-2r$.
    \item {\it Lennard-Jones interaction}. Let $A>0$, $B \in \R$, $R>0$ and $m>n>d$. For any $\bg \in \Ccal_f$, the Hamiltonian is given by 
    \begin{equation}
        H(\bg) = \sum_{\{x,y\} \in \bg } A \left( \frac{R}{|x-y|}\right)^m - B \left( \frac{R}{|x-y|}\right)^n.
    \end{equation}
    This infinite range model satisifies~\ref{H:summable}. Furthermore, for $\Lambda \subset \Lcal$, $\varepsilon< \frac{\delta}{2}$ and $|x_i| \leq \varepsilon$, by integral-series comparison there is $c>0$ such that 
    \begin{equation}
        H(\{i+x_i\}_{i\in\Lambda}) \leq c|\Lambda| \int_{\delta - 2\varepsilon}^\infty r^{d-1} \left| A \left(\frac{R}{r}\right)^m - B\left(\frac{R}{r}\right)^n \right| \d r, 
    \end{equation}
    and thus the interaction satisfies \ref{H:nondegenerate}.
    \item {\it Quermass interaction.} Let $\theta_i \in \R$ for $i \in [0,d]$ and $R>0$. For any $\bg \in \Ccal_f$, the Hamiltonian is given by
    \begin{equation}
        H(\bg) = \sum_{i = 0}^d \theta_i M_i^d(L_R(\gamma))
    \end{equation}
    where $M_i^d$ is the $i$-th Minkowski functional, in particular $M_d^d$ is the Lebesgue measure, $M_{d-1}^d$ is the $d-1$-Hausdorff measure of the surface and $M_0^d$ is the Euler-Poincaré characteristic, and $L_R(\bg)$ is the halo of a configuration, i.e.
    \begin{equation*}
        L_R(\bg) = \bigcup_{x \in \bg} B(x,R).
    \end{equation*}
    When $\theta_i =0 $ for $i=0,\dots,d-1$ and $\theta_d \in \R$, the Quermass interaction corresponds to the one component Widom-Rowlinson model and it satisfies \ref{H:bounded}. When $d=2$, the Quermass interaction verifies \ref{H:stability}, as proven in \citep{kendall1999quermass}. More precisely, they show that there is $c>0$ such that $|M_0^2(L_R(\gamma))|\leq c N(\gamma)$. However, in higher dimensions ($d \geq 3$), the stability of this Hamiltonian remains unclear; see \citep{kendall1999quermass} for further discussion. Furthermore, in dimension 2, for $\Lambda \subset \Lcal$, $\varepsilon >0$ and $|x_i| \leq \varepsilon$ we have
    \begin{equation}
        H\left(\{i+x_i\}_{i \in \Lambda}\right) \leq |\Lambda| \left(c + |M_1^2(B(0,R))| + |M_2^2(B(0,R))|\right)
    \end{equation}
    and thus the interaction verifies \ref{H:nondegenerate}.
\end{enumerate}

\noindent \paragraph{Move (or mark) distribution}The Gibbs perturbed lattice model, precisely described in the next section,  also depends on a move distribution $\mathbb Q$. Again to ensure probabilistic properties of the resulting Gibbs perturbed lattice, we require some of the assumptions from the following list.
\begin{enumerate}[($\mathcal Q$1)]
\setcounter{enumi}{-1} 
\item $\mathbb Q(B(0,\varepsilon))>0$ for the real number $\varepsilon$ given in~\ref{H:nondegenerate}. \label{Q:B0e}
\item {\!\!\![$m$]}. There exists $c>0$ such that  $\E(\e^{c|X|^m})<\infty$ where $X \sim \mathbb Q$. \label{Q:momentexp}
\item The set $\Qcal = \supp(\mathbb{Q})$ is bounded.  \label{Q:bounded}
\item $\E(|X|^d)<\infty$ for $X\sim \mathbb Q$.   \label{Q:moment}
\end{enumerate}

Obviously, if  $\mathbb Q$ has a continuous density on $\Qcal$, $\mathbb Q(B(0,\varepsilon))>0$ is satisfied  for any $\varepsilon$ such that $B(0,\varepsilon)\cap \Qcal \neq \emptyset$. Second, 
\ref{Q:bounded} implies~\ref{Q:momentexp}[$m$] (for any $m>0$) which itself implies \ref{Q:moment}. Three main examples are considered in this paper: (a) The uniform distribution on $\Qcal$ which satisfies \ref{Q:bounded}; (b) the isotropic Gaussian distribution with variance $\sigma^2$, which satisfies \ref{Q:momentexp}[$m$] (for $m=1,2$); (c) The exponential distribution on the positive orthant, i.e. the distribution with density $\lambda_1(x)=\exp(-q(x))$ with $q(x)=  \theta_1 \sum_i x_i + c(\theta_1)$ for any $x\in \Qcal=\{x\in\mathbb R^d: x_i\ge 0, i=1,\dots,d \}$. For this model $p_1=1$, $\theta_1\in \mathbb R^+$ and \ref{Q:momentexp}[1] is satisfied.

%%%%%%%%%%%%%%%%%%%%%%%%%%%%%%%%%%%%%%%%%%%%%%%%%%%%%%%%%%%%%%%%%%%%
%%%%%%%%%%%%%%%%%%%%%%%%%%%%%%%%%%%%%%%%%%%%%%%%%%%%%%%%%%%%%%%%%%%%
%%%%%%%%%%%%%%%%%%%%%%%%%%%%%%%%%%%%%%%%%%%%%%%%%%%%%%%%%%%%%%%%%%%%
%%%%%%%%%%%%%%%%%%%%%%%%%%%%%%%%%%%%%%%%%%%%%%%%%%%%%%%%%%%%%%%%%%%%
%%%%%%%%%%%%%%%%%%%%%%%%%%%%%%%%%%%%%%%%%%%%%%%%%%%%%%%%%%%%%%%%%%%%

\section{Existence and properties of Gibbs perturbed lattices} \label{sec:existence}

\subsection{Gibbs measure on a finite sub-lattice}
Before we introduce and define the Gibbs perturbed lattice, we define a Gibbs measure on a finite sub-lattice with spin space $\R^d$ and a Hamiltonian indexed by the lattice and given by
\begin{equation}
\tilde{H}(\bx_\Lambda)  = H(\bG_\Lambda(\bx_\Lambda, 0))
\end{equation}
for any finite $\Lambda \subset \Lcal$ and any $\bx_\Lambda \in (\R^d)^\Lambda$. 
\begin{definition}
For any configuration $\bx \in (\R^d)^\Lcal$ and any finite subset $\Delta \subset \Lcal$, we define the local energy of $\bx$ moved from $\Delta$ as 
\begin{equation}
\tilde H_\Delta(\bx) = \lim\limits_{n\to \infty} \left( \tilde H(\bx_{\Lambda_n}) - \tilde H (\bx_{\Lambda_n \setminus \Delta}) \right)
\end{equation}
where $\Lambda_n = [-n,n]^d \cap \Lcal$  and with the convention $\infty - \infty =0$.
\end{definition}
It is worth pointing out that if $H$ has a finite range $R$, i.e. satisfies~\ref{H:range}, $\tilde H$ has a finite range if and only if $\Qcal$ is bounded, i.e. satisfies~\ref{Q:bounded}. In such a case, the finite range of $\tilde H$ is $\tilde R= R + \diam \Qcal$ and for any finite $\Delta \subset \Lcal$ we have
\begin{equation}
\tilde{H}_\Delta(\bx) = \tilde{H}_\Delta(\bx_{\Delta \oplus B(0,\tilde R) \cap \Lcal }).
\end{equation}
%In the following section, we define and prove the existence of the Gibbs measure on the lattice underlying the model of the Gibbs perturbed lattice. First, we define the Gibbs perturbed finite sub-lattice.
\begin{definition}
Let $\Lambda \subset \mathcal{L}$ finite, the distribution of the Gibbs measure on the sub-lattice $\Lambda$ is given by
\begin{equation}
\P_\Lambda(\d \bx_\Lambda) = \frac{1}{Z_\Lambda} \e^{-\tilde H(\bx_\Lambda)} \mathbb{Q}^{\otimes \Lambda}(\d \bx_\Lambda)
\end{equation}
where 
$$
Z_\Lambda=\int \e^{-\tilde H_\Delta(\bx_\Delta)  } \mathbb{Q}^{\otimes \Delta}(\d \bx_\Delta)
$$
is the normalizing constant, also called the partition function.
\end{definition}
Under the assumptions \ref{H:stability}-\ref{H:nondegenerate} we have $(\log\mathbb{Q}(B(0,\varepsilon))-B) |\Lambda| \le \log Z_\Lambda| \leq  A |\Lambda|)$,
%\begin{align*}
%\e^{(\log\mathbb{Q}(B(0,\varepsilon))-B) |\Lambda|}\leq Z_\Lambda \leq \e^{A |\Lambda|}, 
%\end{align*}
which ensures that the Gibbs measure on the sub-lattice is well-defined. The next result states that it also satisfies the Dobrushin-Lanford-Ruelle (DLR) equations, i.e. provides the conditional distribution of the lattice for any finite subset $\Delta \subset \Lambda$ given the configuration outside of $\Delta$.
\begin{proposition}[DLR Equations]
Let $\Delta \subset \Lambda \subset \Lcal$ be finite, we have for $\P_\Lambda$-a.e $\bx_{\Lambda \setminus \Delta}$,
\begin{equation}
\P_\Lambda (\d \bx_\Delta' | \bx_{\Lambda \setminus \Delta}) = \frac{1}{Z_\Delta(\bx_{\Lambda \setminus \Delta})}  \e^{-\tilde H_\Delta(\bx_\Delta' \cup \, \bx_{\Lambda \setminus \Delta})  } \mathbb{Q}^{\otimes \Delta}(\d \bx_\Delta^\prime)
\end{equation}
where $Z_\Delta(\bx_{\Lambda \setminus \Delta})$ is the normalizing constant given by
\begin{equation*}
Z_\Delta(\bx_{\Lambda \setminus \Delta}) = \int \e^{-\tilde H_\Delta(\bx_\Delta' \cup \, \bx_{\Lambda \setminus \Delta})  } \mathbb{Q}^{\otimes \Delta}(\d \bx^\prime_\Delta).
\end{equation*}
\end{proposition}

\begin{proof}
By definition of $H_\Delta$ we have for any $\bx_\Lambda$, $\tilde H_\Delta(\bx_\Lambda) = \tilde H(\bx_\Lambda) - \tilde H(\bx_{\Lambda \setminus \Delta})$ and thus
\begin{align*}
\P_\Lambda(\d \bx_\Lambda) = \frac{1}{Z_\Lambda} \e^{- \tilde H_\Delta(\bx_\Lambda) - \tilde H(\bx_{\Lambda \setminus \Delta})} \mathbb{Q}^{\otimes \Delta}(\d \bx_\Delta) \mathbb{Q}^{\otimes \Lambda \setminus \Delta}(\d \bx_{\Lambda \setminus \Delta})
\end{align*}
whereby we deduce that the density of the conditional probability $\P_\Lambda(\d \bx_\Delta' | \bx_{\Lambda \setminus \Delta})$ with respect to $\mathbb{Q}^{\otimes \Delta}$ is $:\bx_\Delta' \to \e^{- \tilde H_\Delta(\bx_\Delta' \cup \, \bx_{\Lambda \setminus \Delta})}$ and $Z_\Delta(\bx_{\Lambda \setminus \Delta})$ is necessarily the normalizing constant.
\end{proof}

\subsection{Infinite Gibbs perturbed lattice}

In the previous section, Gibbs measures on finite sub-lattices are shown to satisfy the DLR equations. These equations remain key ingredients in the context of an infinite lattice.
\begin{definition}
A probability measure $\P$ on $(\R^d)^\Lcal$ invariant under translation on $\Lcal$ is a Gibbs measure for the Hamiltonian $H$ if it satisfies the DLR equations for any finite $\Delta \subset \Lcal$, i.e. for $\P$-a.e. $\bx_{\Delta^c}$,
\begin{equation}\label{eq:DLR}
\P(\d \bx_\Delta' | \bx_{\Delta^c}) = \frac{1}{Z_\Delta(\bx_{\Delta^c})}  \e^{-\tilde H_\Delta(\bx_\Delta' \cup \, \bx_{\Delta^c})  } \mathbb{Q}^{\otimes \Delta}(\d \bx_\Delta^\prime)
\end{equation}
where 
\begin{equation*}
Z_\Delta(\bx_{\Delta^c}) = \int  \e^{-\tilde H_\Delta(\bx_\Delta' \cup \, \bx_{\Delta^c})  } \mathbb{Q}^{\otimes \Delta}(\d \bx_\Delta^\prime).
\end{equation*}
We denote by $\Gcal_\Lcal(H,\mathbb Q)$ the set of all Gibbs measures. 
\end{definition}

It is worth pointing out that if $\Gcal_\Lcal(H, \mathbb Q)$ is non-empty, then any Gibbs measure $\P$ on the lattice is by definition invariant under translation on  $\Lcal$. Therefore, if $\bX = \{X_i, i\in \L\}\sim \P$, then $(X_i)_{i \in \Lcal}$ are identically distributed. 

\begin{definition}
Let $\P \in \Gcal_\Lcal(H,\mathbb Q)$, the Gibbs perturbed lattice is defined as $\bG := \bG(\bX, U)$ where $\bX \sim \P$, $U \sim \mathcal{U}_\Dcal$ and $\bX$ is independent of $U$. We denote by $P_\Lcal(H,\mathbb Q)$ the set of stationary Gibbs perturbed lattices. 
\end{definition}

We now present our main result and provide assumptions on the point process Hamiltonian $H$ and on the mark distribution $\mathbb{Q}$ ensuring  the existence of a solution to DLR equations. 

\begin{theorem}\label{thm:existence} Assume \ref{H:stability}-\ref{H:invariant}. Assume also either \ref{H:range}-\ref{Q:momentexp}[\!1] or \ref{H:summable}-\ref{Q:bounded}, then $\Gcal_\Lcal(H, \mathbb Q) \neq \emptyset$, $P_\Lcal(H,\mathbb Q) \neq \emptyset$ and any Gibbs perturbed lattice $\bG$ is a stationary point process on $\R^d$. Moreover, for any $u\in \R^d$, the distribution of $\bG$ given $U=u$, denoted by $\P_u$, is invariant under any translation on $\Lcal$.
\end{theorem}

The proof of Theorem~\ref{thm:existence} relies on  specific-entropy techniques developed in \citep{georgii2011book} to establish the existence of a limiting process. We adapt parts of this approach to show that the limit satisfies the DLR equations. The details are provided in~\ref{app:aux.existence}–\ref{app:existence}. It is worth mentioning that, by construction, the set $P_\Lcal(H,\mathbb{Q})$ is a Choquet simplex and any Gibbs perturbed lattice is a convex combination of extremal ergodic measures. This property follows from the fact that $\Gcal_\Lcal(H,\mathbb Q)$ is a Choquet simplex; see Chapter 14 of \citep{georgii2011book} for further details. Thus, $\P$ and $\P_u$ can be decomposed as mixtures of ergodic measures.

\subsection{Hyperuniformity of Gibbs perturbed lattices}

A point process $\bG$ is said to be hyperuniform if, as $R\to \infty$,
\begin{equation}
\frac{\Var[\# \bG \cap B(0,R) ]}{|B(0,R)|} \rightarrow 0
\end{equation}
where $B(z,r)$ denotes the Euclidean ball centered at $z\in \R^d$ with radius $r>0$. Hyperuniformity of perturbed lattices has been investigated in~\citep{dereudre2024non2}. Their Theorem~1, in particular, establishes the hyperuniformity of any point process obtained from a perturbed lattice as soon as $\E(|X_0|^d)<\infty$, where $X_0$ is the mark from a typical location on the lattice, say $0$ by stationarity. It is worth emphasizing that the global shift $U$ is necessary to establish the hyperuniformity of the perturbed lattice. The next result investigates this condition for Gibbs perturbed lattices.

\begin{proposition}\label{prop:HU} We have the following two statements.\\
(i) Assume \ref{H:stability}-\ref{H:invariant}.
We also assume either \ref{H:range}-\ref{Q:momentexp}[\!d] or   \ref{H:summable}-\ref{Q:bounded}, then there exists $\P \in \Gcal_\Lcal(H, \mathbb Q)$ and  $\bX \sim \P$ such that $\E(|X_0|^d)<\infty$. This implies that when $d\le2$, there exists a hyperuniform point process $\bG \in P_\Lcal(H, \mathbb Q)$.\\
(ii) Assume \ref{H:stability}-\ref{H:invariant} and \ref{H:bounded}-\ref{Q:moment}, then for any $\P \in \Gcal_\Lcal(H, \mathbb Q)$ and  $\bX \sim \P$, we have $\E(|X_0|^d)<\infty$. This implies that when $d\le2$, any $\bG \in P_\Lcal(H, \mathbb Q)$ is hyperuniform.
\end{proposition}

The key ingredient for the proof of Proposition \ref{prop:HU}, proposed in~\ref{app:hyperuniformity}, relies on the variational principle that requires the specific entropy to be finite; see \eqref{ineq: bounded specific entropy}. What distinguishes (i) from (ii) is precisely the fact that  specific entropies are not necessarily finite for every Gibbs measure.

%Under the assumptions of Proposition~\ref{cor:HU}, the fact that $\E(|X_0|^d)$ is not necessary finite for all Gibbs measures stems from the fact that the specific entropy might not be finite for any $\P \in \Gcal_\Lcal(H, \mathbb Q)$. However, if we were assuming that the variational principle \JF{Ref à un truc dans l'annexe?} holds, Proposition~\ref{cor:HU} would be true for any Gibbs measure. It turns out that a stronger assumption on the Hamiltonian (but a less restrictive assumption on the move distribution), this variational principle can be proved yielding the following result.

\subsection{DLR type equations} \label{sec:dlr}

The Gibbs perturbed lattice satisfies several equations that are derived from the DLR equations \eqref{eq:DLR}. These equations can be derived not only for the expectation with respect to the Gibbs measure but also for the conditional expectation given $U=u \in \Dcal$, denoted by $\Eu[\cdot]$ in the rest of the paper.

We start with a presentation of the DLR equation at a single site $i \in \Lcal_U$ by applying equation~\eqref{eq:DLR} with $\Delta=\{i\}$. Let $f:\Ccal \to \R$ be a  bounded measurable function. Then, we have 
%M\JF{on n'écrirait pas plutôt: 
%$\E_u[ f(i,X_i,\bG_{i^c})]$ et d'ailleurs c'est aussi une espérance conditionnelle par rapport à $\bG_{i^c}$}
\begin{equation}\label{eq:onesiteDLR}
\E_u [f(\bG)] = 
\E_u 
\left[ 
\frac{1}{Z_{i}(\bG_{i^c})} \int f( \bG_{i^c} \cup \{ i + x + U\} ) \e^{-h(i+x+ U, \bG_{i^c})} \mathbb Q (\d x)
\right]
\end{equation}
where for $i\in \Lcal_U=\{ i + U, i \in \Lcal \}$ and $\Lcal_U$ is the shifted lattice,
\begin{equation}\label{eq:Gic}
\bG_{i^c}=\{j+X_j, j\in \Lcal_U, j \neq i\}.
\end{equation}
In the following, we assume that $\mathbb Q$ is absolutely continuous with respect to the Lebesgue measure and we denote by $\lambda_1(x)=\e^{-q(x)}$ its density with support  $\Qcal$. We also denote, for any $i\in \Lcal_U$, $x\in \Qcal$ and $\bg \in \Ccal$, $\lambda_2(i+x,\bg)= \e^{-h(i+x,\bg)}$, the two measurable functions 
\begin{equation} \label{eq:lambda}
 \lambda (i,x,\bg) :=  \lambda_1(x) \lambda_2(i+x,\bg)% = \e^{-q(x) -h(i+x,\bg)}  
 \quad \text{ and } \quad 
 \Lambda(i,x,\bg) = \frac{ \lambda (i,x,\bg) }{Z_i(\bg)} = \frac{ \lambda (i,x,\bg) }{\int_{\Qcal} \lambda(i,x,\bg)\d x}.
\end{equation}
These functions respectively act as unnormalised and normalised Papangelou conditional intensities. 
We are now in a position to state the main equations used in Section~\ref{sec:inference} to derive parametric inference methods. The next three results are proved in~\ref{app:DLR}.

\begin{proposition} \label{prop:sumDLR}
Assume the conditions of Theorem~\ref{thm:existence} are satisfied and thus let $\bG \sim \P \in P_\Lcal(H,\mathbb Q) \neq \emptyset$. For  $u \in \Dcal$, let  $f:\Lcal_u \times \R^d \times \Ccal \to \R$ be  any measurable function.  

(i) Assume that
\begin{equation*}
\E_u  \sum_{i \in \Lcal_U} |f(i,X_i, \bG_{i^c})|<\infty 
\quad \text{ and } \quad
\E_u \int_\Qcal |f(i,x, \bG_{i^c})| \Lambda(i,x, \bG_{i^c}) \d x  < \infty,
\end{equation*}
then  we have
\begin{align}
\E_u \sum_{i \in \Lcal_U} f(i,X_i, \bG_{i^c})  &= \E_u  \sum_{i \in \Lcal_U} \int_\Qcal f(i,x,\bG_{i^c}) \Lambda(i,x, \bG_{i^c}) \d x . \label{eq:sumDLR}
\end{align}
(ii) Assume that $f$ is such that $f(i,x,\bg)=f(0,x,\bg-i)$ for any $i\in \Lcal_u$, $x\in \Qcal$, $\bg\in \Ccal$ and such that 
\begin{align*}
    \Eu  \int_{\Qcal} |f(0,x,\bG_{0^c})| \Lambda(0,x,\bG_{0^c}) \d x <\infty,
\end{align*}
then, for any bounded domain $W \subset \R^d$, we have
\begin{align}
    \E_u  \sum_{i \in \Lcal_U \cap W} f(i,X_i, \bG_{i^c})  &= \E_u  \sum_{i \in \Lcal_U \cap W} \int_\Qcal f(i,x,\bG_{i^c}) \Lambda(i,x, \bG_{i^c}) \d x.  
\end{align}
%where $f_W(i,x,\bg)=\1_W(i+x) f(i,x,\bg)$ for any $i \in \Lcal_u,x\in \Qcal$ and $\bg\in \Ccal$. \JF{Stop here: à vérifier la "edge" correction ici. Est-ce intéressant?pertinent? ou alors utiliser celle utilisée en stat}
\end{proposition}

Starting from the DLR equations on two sites, we can obtain a similar second-order GNZ-type equation \citep{georgii1979canonical,georgii2011book,nguyen1977integral}.

\begin{proposition} \label{prop:sum2sitesDLR}
Assume the conditions of Theorem~\ref{thm:existence} are satisfied and thus let $\bG \sim \P \in P_\Lcal(H,\mathbb Q) \neq \emptyset$. For any $u \in \Dcal$ and any measurable function $f:\Lcal_u^2 \times (\R^d)^2 \times \Ccal \to \R$ such that
\begin{align*}
\E_u & \sum_{\{i,j\} \in \Lcal_U} \left| f(i,j,X_i, X_j, \bG_{\{i,j\}^c}) \right| <\infty\\
\E_u  & \int_\Qcal \int_\Qcal \left| f(i,j,x, y, \bG_{\{i,j\}^c}) \right| \Lambda(i,x, \bG_{\{ i,j\}^c}) \Lambda(j,y, \bG_{\{ i,j\}^c} \cup \{i+x\}) \d x \d y  < \infty.
\end{align*}
Then, we have
\begin{align*}
\E_u &\sum_{\{i,j\} \in \Lcal_U}  f(i,j,X_i, X_j, \bG_{\{i,j\}^c})  =\\
&\E_u   \sum_{\{i,j\} \in \Lcal_U} \iint_{\Qcal^2} f(i,j,x, y, \bG_{\{i,j\}^c}) \Lambda(i,x, \bG_{\{ i,j\}^c}) \Lambda(j,y, \bG_{\{ i,j\}^c} \cup \{i+x\}) \d y \d x. 
\end{align*}
\end{proposition}

To complete these DLR-type equations, we now propose a variational equation, similarly to the one obtained by~\citep{baddeley2013variational}. This requires a few additional definitions and notation. Let $\check f:\R^d \times \Ccal\to \R$ be a measurable function and $x \in \R^d$. We say that $\check f$ is differentiable at $x$, if for every $\bg \in \Ccal$ the function $:x \mapsto \check f(x,\bg)$ is differentiable at $x$. Furthermore, $\check f$ is $\lambda^d$-a.e. differentiable, if $\check f$ is differentiable for $\lambda^d$-a.e. $x \in \R^d$. We denote by $\grad \check f(x,\bg)$ the corresponding gradient. By convention, if $\check f(y) = \infty$ for $y \in \mathcal{V}$, where $\mathcal{V}$ is an open neighborhood of $x$, we have $\grad \check f(x, \bg) = 0$. Let $\check f,h:\R^d \times \Ccal\to\R$ be measurable functions. We say that $\check f$ is regularizing with respect to $h$ if for every $\bg \in \Ccal$ the function 
\begin{equation*}
    x \mapsto \check f(x,\bg)\e^{-h(x,\bg)}
\end{equation*}
is absolutely continuous. As a consequence, if $\check f$ and $h$ are $\lambda^d$-a.e. differentiable and $\check f$ is regularizing with respect to $h$, then we have
\begin{equation}
    \grad \left( \check f \e^{-h} \right)(x, \bg) = \left( \grad \check f(x,\bg) - \check f(x,\bg) \grad h(x,\bg) \right)\e^{-h(x, \bg)}.
\end{equation}
Finally, for $u\in \Dcal$, we denote by $\Fcal_\Rcal(H, \mathbb{Q}, \Lcal_u)$ the class of functions that are $\lambda^d$-a.e. differentiable, regularizing with respect to $h + q$, and such that for any $x \in \partial \Qcal$, $i \in \Lcal_u$ and $\bg \in \Ccal$, we have $\check f(x+i,\bg) = 0$ and $\check f$ vanishes at infinity. 
\begin{proposition} \label{prop:variational}
Assume the conditions of Theorem~\ref{thm:existence} are satisfied and thus let $\bG \sim \P \in P_\Lcal(H,\mathbb Q) \neq \emptyset$. For  $u \in \Dcal$ and $\check f \in \Fcal_\Rcal(H, \mathbb{Q}, \Lcal_u)$  such that 
\begin{align}\label{condition: fubini_eq_variationnel}
\E_u  \sum_{i \in \Lcal_U} \Big\{
|\grad \check f(i+X_i, \bG)| + |\check f(i+X_i, \bG)| \left| \grad h(i+X_i, \bG_{i^c}) + \grad q(X_i) \right|  \Big\}< \infty,
\end{align}
then we have the variational equation
\begin{align}\label{eq: variationnel}
\E_u  \sum_{i \in \Lcal_U} 
\Big\{
\grad \check f(i+X_i, \bG) - \check f(i+X_i, \bG) \left( \grad h(i+X_i, \bG_{i^c}) + \grad q(X_i) \right) 
\Big\}= 0.
\end{align}
\end{proposition}

\begin{remark}\label{rem:dlr}
It is worth mentioning that Proposition~\ref{prop:variational}, and in particular the variational equation~\eqref{eq: variationnel}, can be seen as a specific case of Proposition~\ref{prop:sumDLR} by taking $f(i,x,\bg) = \grad \check f(i+x,\bg) - \check f(i+x,\bg) (\grad h(i+x,\bg) + \grad q(x_i))$.
\end{remark}

\section{Parametric inference of Gibbs perturbed lattices} \label{sec:inference}

\subsection{Notation and parametric model}

In this section, we assume that assumptions \ref{H:stability}-\ref{H:invariant} and \ref{H:range}-\ref{Q:momentexp}[\!1] are satisfied and thus, by Theorem~\ref{thm:existence}, there exists $\bG\sim \mathbb P \in P_\Lcal(H, \mathbb Q)$. Assuming the finite range property \ref{H:range} eases in many ways the statistical methodology and its asymptotic properties. We leave for future research the situation where the Hamiltonian satisfies \ref{H:summable}, as is the case, for example, for the Lennard-Jones potential, as considered by~\citep{coeurjolly2010asymptotic,coeurjolly2017parametric} for continuous Gibbs point processes.  

We recall the notation $\bG = \{i+X_i, i\in \Lcal_U\}$, where $X_i$ are realizations of the mark distribution $\mathbb Q$, assumed to have a density $Q$ with known support $\mathcal Q$ (possibly $\mathbb R^d$). Finally, we also remind that the notation  $\bG_{i^c}$, given by~\eqref{eq:Gic}, stands for the set of perturbed points except the one coming from $i$.  For any $u\in \Dcal$ and any function $f: \Lcal_u \times  \Qcal \times \Ccal \to \mathbb R$, we consider the following properties and definitions: 
\begin{itemize}
\item[\namedlabel{TI}{(TI)}] $f$ is said to be invariant under translation if for any $i\in \Lcal_u$, $x\in \Qcal$ and $\bg\in \Ccal$, $f(i,x,\bg)=f(0,x,\bg-i)$. 
\item[\namedlabel{FR}{(FR)}] $f$ is said to have a finite range $R>0$ if for any $i\in \Lcal_u$, $x\in \Qcal$ and $\bg\in \Ccal$, $f(i,x,\bg)=f(i,x,\bg_{B(i+x,R)})$. 
\end{itemize}

In particular, the local energy function $\tilde h(i,x,\bg)=h(i+x,\bg)$ as well as $\lambda_2(i+x,\bg)$ satisfy \ref{TI}-\ref{FR}.

This section focuses on parametric modelling and estimation: for any $u \in \Dcal, i\in\Lcal_u, x\in \Qcal$, $\bg \in \Ccal$, we let $\lambda(i,x,\bg;\theta)$ (resp. $\Lambda(i,x,\bg;\theta)$) denote a parametric version of~\eqref{eq:lambda}, where $\theta \in \Theta \subseteq \mathbb R^p$ for some $p\ge 1$.  Hence, we consider parametric forms of $q$ and the local energy function $h$. To ease statistical methodologies and asymptotic results, it is standard to focus on exponential family models \citep{jensen1991pseudolikelihood}. So, on the one hand, we assume that $-\log \lambda_1(x)=q(x)$ is of the form $\theta_1^\top S_1(x)+ c(\theta_1)$ where $\theta_1 \in \mathbb R^{p_1}$ for some $p_1\ge 1$ and $c(\theta_1)$ is the log-partition constant. Here are a few examples:
\begin{itemize}
\item Uniform distribution on $\Qcal$ (with $|\Qcal|<\infty$). In this case, there is no parameter to estimate and we take the convention $p_1=0$.
\item Centered Gaussian density on $\Qcal=\mathbb R^d$: $p_1=1$, $\theta_1\in \R^+$ and $S_1(x)= \|x\|^2$ and $c(\theta_1)=-d/2\log(\pi\theta_1)$.
\item Exponential distribution on $\Qcal=\{x\in\mathbb R^d: x_i\ge 0, i=1,\dots,d \}$: $p_1=1$, $\theta_1\in \mathbb R^+$, $S_1(x)=\sum_{i=1}^d x_i$ and $c(\theta_1)=d\log(\theta_1)$.
\end{itemize}
On the other hand, we assume that the local energy function $h=-\log \lambda_2$ can be written as $-\log \lambda_2(i+x,\bg) = h(i+x,\bg)= \theta_2^\top S_2(i+x, \bg)$ where $\theta_2\in \mathbb R^{p_2}$, $p_2\ge 1$. By assumption,  $\lambda_2$, $h$ and $S_2$ satisfy \ref{TI}-\ref{FR}.

As a first example, let us cite the well-known Strauss hard-core model for which $p_2=1$ and 
\begin{equation*}
S_2(y,\bg)=\sum_{y^\prime \in \bg} 
\left\{
\1_{[r,R]} (| y'-y|) + (\infty)\times  \1_{[0,r)}(|y^\prime-y|)
\right\}
\end{equation*}
for any $y\in \R^d$ and known real parameters $r,R$ such that $0\leq r < R$. The exponential family models also include for instance piecewise versions of the Strauss hard-core model, the Geyer saturation model and the Widom-Rowlinson area-interaction  \citep{dereudre2014estimation,illian2008statistical,moller2003statistical}. Note that all these models have finite range. Moreover, they also all satisfy some local stability property which means that for any $y\in \R^d$ and $\bg \in \Ccal$, $\e^{-h(y, \bg)}\le \kappa$ for some $\kappa>0$ (independent of $y,\bg$).

As a summary let $\theta = (\theta_1^\top , \theta_2^\top)^\top\in \mathbb R^p$ and for any $x\in \mathcal Q$ , $S(i,x,\bg) := (S_1(x)^\top ,S_2(i+x,\bg)^\top)^\top $ with $p=p_1+p_2$. All in all, we consider the following exponential parametric model
\begin{align} 
- \log \lambda(i,x,\bg; \theta)  &= -\log \lambda_1(x;\theta) - \log \lambda_2(i+x,\bg; \theta) \hspace*{-1cm}&=& \; h(i+x,\bg)  +q(x) \nonumber\\
&=\theta_1^\top S_1(x) + \theta_2^\top S_2(i+x,\bg)&=&\; \theta^\top S(i,x,\bg). \label{eq:explambda}
\end{align}
Finally, we denote by $\theta^\star$ the true parameter vector to be estimated, i.e. we assume that $\bG \sim \mathbb P = \mathbb P_{\theta^\star}$.

\subsection{Statistical context and methodologies}

In this paper, we consider two different statistical settings, illustrated by Figure~\ref{fig:framework}. Let $W_n$ be a bounded domain of $\mathbb R^d$ which plays the role of the observation domain.
\begin{itemize}
\item {\bf Framework 1}. We observe the shifted lattice $\Lcal_U$ and all perturbed points falling in $W_n$, i.e. we observe a realization of
\begin{equation} \label{eq:fwk1}
\Dcal_{1,n} = \{(i,X_i), \text{ with } i\in \Lcal_U\cap W_n \text{ and } i+X_i\in W_n\}.
\end{equation}
\item  {\bf Framework 2}. We observe the shifted lattice $\Lcal_U$, and only the perturbed points falling in $W_n$, i.e. we observe a realization of
\begin{equation} \label{eq:fwk2}
\Dcal_{2,n} = \{i+X_i \in W_n \text{ for } i\in \Lcal_U  \}.
\end{equation}
\end{itemize}

\begin{figure}[htbp]
\centering
\begin{tabular}{cc}
\includegraphics[width=.4\textwidth]{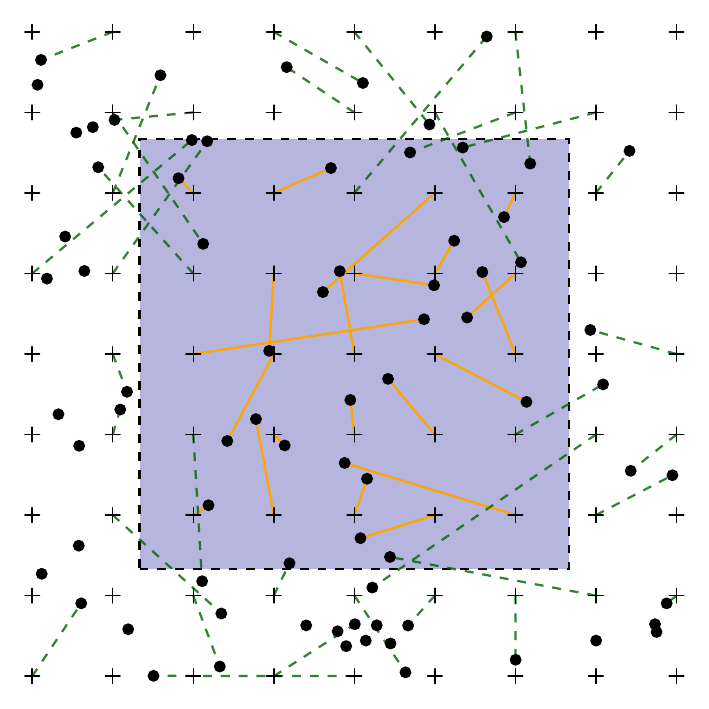}&
\includegraphics[width=.4\textwidth]{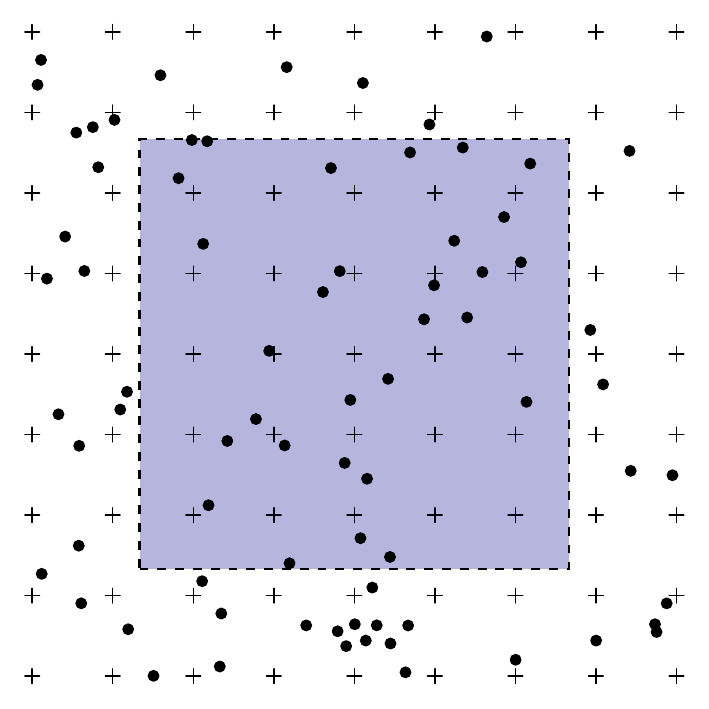}\\
(a) Framework 1 & (b) Framework 2
\end{tabular}
\caption{Illustration of the two statistical frameworks considered in this paper. In (a)-(b) the inside grey rectangle corresponds to $W_n$, crosses to points $i\in \Lcal_U$ and dots to $i+x_i$, for $i\in \Lcal_U$. For (a) Framework~1, we observe crosses ($i$) and dots ($i+x_i$) that fall in $W_n$. For (b) Framework~2, we observe only dots falling in $W_n$. 
%, illustrated by orange segments. Inside $W_n$, there exist lattice points of $\Lcal_U$ for which their perturbations fall outside $W_n$ and perturbed points in $W_n$ (i.e. points) for which their origins on the lattice are outside $W_n$. The situation is somehow simpler to describe for (b) Framework 2, we observe $\Lcal_U$ and only locations of the dots, i.e. $i+x_i$, falling in $W_n$.
}
\label{fig:framework}
\end{figure}

The rest of this section discusses strategies to estimate $\theta^\star$ based on a single observation of either $(i,x_i)_i$ or $(i+x_i)_i$ on a bounded domain $W_n \subset \mathbb R^d$. The domain $W_n$ plays the role of the observation domain and grows to $\mathbb \R^d$ as $n\to \infty$. We also let $W_n\ominus r$ for some $r\ge0$ be the domain $W_n$ eroded by $r$. The standard approach would consist in using the likelihood method. However, due to the unknown normalizing constant depending on the parameter, it is well-known  that the likelihood method is already computationally expensive even for estimating an unmarked  Gibbs model on a lattice \citep{gaetan2010spatial}. Many alternative methods have been proposed in the literature for Gibbs models on the lattice or in the continuous case, which are in particular appropriate for Framework~1~\citep{gaetan2010spatial}. Unlike standard Gibbs lattice models, it is worth noting that this model has a continuous mark distribution and has the particularity that the Hamiltonian depends only on $i+X_i$, $i\in \Lcal_U$ and $X_i\in \Qcal$. The most popular alternative method is the ($\log$-)pseudolikelihood. To account for edge-effects problems, we have to introduce a specific border correction factor, $b_n$, described below. The border-corrected ($\log$)-pseudolikelihood we consider is given by
\begin{align}
    \mathrm{LPL}_{n}(\bG;\theta) &= \sum_{i \in \Lcal_U }\, \log \Lambda_{n}(i,X_i, \bG_{i^c};\theta) ,\label{eq:LPL}
\end{align}
where  for any $i\in \Lcal_U, x\in \Qcal$ and $\bg \in \Ccal$
\begin{equation} \label{eq:Lambdan}
    \Lambda_{n}(i,x,\bg;\theta) = \frac{b_n(i,x) \e^{-\theta^\top S(i,x,\bg)}}{\int_{\Qcal} {b}_n(i,y)\e^{-\theta^\top S(i,y,\bg)} \d y} 
\end{equation}
where the term $b_n$ is called a border correction factor. For technical reasons, mainly explained in Theorem~\ref{thm:ergodic}(ii), we define this border correction factor as
\begin{equation}
\label{eq:bn}    
b_n(i,x) = \1(i\in W_n \ominus(m_n+R), i+x \in W_n\ominus R)
\end{equation}
where $(m_n)_n$ is a sequence of real numbers, specified later in assumption~\ref{M:model}. It is aimed to tend to $\infty$ such that $|W_n\ominus(m_n+R)|\sim |W_n|$ as $n\to \infty$. The interest of the eroded domain $W_n\ominus R$ is also revealed. By the finite range assumption, for any $i \in \Lcal_u$ such that $i+z \in W_n\ominus R$ for $z=X_i,x \in \Qcal$, $\Lambda(i,z, \bG_{i^c};\theta)=\Lambda(i,z, \bG_{i^c} \cap{W_n};\theta)$ can always be evaluated with data from Framework~1. The gradient (with respect to $\theta$), denoted by $\mathrm{LPL}_{n}^{(1)}(\bG;u,\theta)$ of  $\mathrm{LPL}_{n}(\bG;u,\theta)$ for exponential family models is given by
\begin{align}
\mathrm{LPL}_{n}^{(1)}(\bG;\theta)     =& \sum_{i \in \Lcal_U} \Big\{ 
    - b_n(i,X_i) S(i,X_i,\bG_{i^c}) + \nonumber\\ 
& \qquad  \qquad    \int_{\Qcal} b_n(i,x) S(i,x,\bG_{i^c}) \Lambda_n(i,x,\bG_{i^c}; \theta)
    \d x\Big\} \label{eq:gradLPL}.
\end{align}
At that point, it is worth observing that, under appropriate assumptions, not detailed for now, and replacing $\Lambda_n$ by $\Lambda$ in the previous equation,
Proposition~\ref{prop:sumDLR} can be applied and would show that~\eqref{eq:gradLPL} is centered  when $\theta=\theta^\star$. 

Let us now consider the situation of Framework~2. Since we lose the links between $i$ and $i+X_i$, it is straightforwardly seen that some elements of~\eqref{eq:gradLPL} cannot be evaluated: (a) for $i+X_i\in \Gamma$, $q(X_i)$ and consequently  $S(i,X_i,\bG_{i^c})$ are not observed; (b) for any $x\in \Qcal$,  $S(i,x,\bG_{i^c})$ is also incomputable since $\bG_{i^c}$ is not observed. In this situation we did not come up with a general estimating equation similar to~\eqref{eq:gradLPL} that would overcome problems mentioned in~(a)-(b). However, for some particular cases, the variational equation~\eqref{eq: variationnel}, allows us to partially solve the problem. Let ${\check f}=(\check f_1,\dots,\check f_L)$, for some $L \ge p$, be a set of functions satisfying the assumptions of Proposition~\ref{prop:variational}, and consider the empirical version of~\eqref{eq: variationnel} given, for the whole set of test functions ${\check f}$ and for exponential family models, by
\begin{align}
&\mathrm{VARE}_{ n}({\check f};\theta) =\nonumber\\
&\quad \sum_{\substack{ i \in \Lcal_U \\ i+X_i \in W_n\ominus R} } 
\left\{
\grad {\check f}(i+X_i; \bG_{i^c}) -{ \check f}(i+X_i; \bG_{i^c}) 
\left( 
\grad S_1(X_i)^\top \theta_1 + \grad S_2(i,X_i,\bG_{i^c})^\top \theta_2
\right)
\right\}. \label{eq:evariational}
\end{align}
If one considers Framework~2, the main problem of~\eqref{eq: variationnel} lies in $\grad S_1$, which is not observable. However, we can exhibit two particular examples, for which this term either cancels out or is equal to one:
\begin{itemize}
\item[(i)] The distribution $\mathbb Q$ is the uniform distribution on $\Qcal$, then in this case $p_1=\grad S_1=0$ and we propose to take $L=p$ and consider ${\check f}(i+x,\gamma)= \grad S_2(i+x,\gamma) \psi(i+x,\gamma)$ with $\psi:\R^d\to \R$.
\item[(ii)] The density $\lambda_1(x)=\e^{-q(x)}$ corresponds to the exponential distribution on the positive orthant described in Section~\ref{sec:hamiltonian}. For this specific distribution, $\grad S_1 = 1$ and we propose to take $L=p$ and consider ${\check f}(i+x,\gamma)= (1+ \grad S_2(i+x,\gamma) )\psi(i+x,\gamma)$ with $\psi:\R^d\to \R$.
\end{itemize}
The variational equation has also the advantage for exponential models to be linear in $\theta$ which yields an explicit and simple estimator. Remark that for both cases (i)-(ii), one can always find a smooth function $\psi$ such that given $U=u$, ${\check f} \in \Fcal_\Rcal(H, \mathbb{Q}, \Lcal_u)$. Hence,  Proposition~\ref{prop:variational} can be applied which also means that given $U=u$~\eqref{eq: variationnel} is an unbiased estimating equation, i.e. 
$\E_u[ \mathrm{VARE}_{n}({\check f};\theta) ]=0$ when $\theta=\theta^\star$.

From a practical point of view, methodologies summarized by~\eqref{eq:gradLPL} and~\eqref{eq:evariational} appear to be different since the first one (resp. second one) is more appropriate for Framework~1 (resp. Framework~2) and that edge effects are treated differently. Nevertheless, from a theoretical point of view, by Remark~\ref{rem:dlr}, estimates obtained from these equations follow from the same DLR-type equation~\eqref{eq:DLR}. Therefore to encompass both methods from a theoretical point of view, we propose to consider the Takacs-Fiksel \citep{coeurjolly2012takacs,fiksel1984estimation,takacs1986estimator} methodology for which \eqref{eq:gradLPL}  and~\eqref{eq:evariational} are particular cases.

Let $L \ge p$ and for any $u\in \Dcal$, let $ \bff : \Lcal_u\times \Qcal \times \Ccal \times \R^p\to \R^L$. Then, we construct the least squares criterion, called Takacs-Fiksel criterion 
\begin{align}
\mathrm{TF}_{n} (  \bff, \bG ; \theta) &= \sum_{l=1}^L \mathrm{DLR}_{n}(f_l,\bG;\theta)^2 \quad \text{ with } \label{eq:TF}\\
\mathrm{DLR}_{n}(f_l,\bG ; \theta) &=\sum_{ \substack{i \in \Lcal_U} } \Big\{ b_n(i,X_i)f_{l}(i,X_i, \bG_{i^c}; \theta ) \nonumber\\
&\qquad \qquad - \int_\Qcal b_n(i,x) f_{l}(i,x, \bG_{i^c} ; \theta) \Lambda_n(i,x, \bG_{i^c} ; \theta) \d x \Big\} \label{eq:DLRfj}
\end{align}
The Takacs-Fiksel estimator is defined as 
$\hat \theta= \mathrm{argmin}_\theta \mathrm{TF}_{n} (  \bff, \bG ; \theta)$. Again, $\hat \theta$ corresponds to the pseudo-likelihood estimator~\eqref{eq:LPL} when $f(i,x,\gamma;\theta)= S(i,x,\gamma)$ and to the variational estimator~\eqref{eq:evariational} for the two particular cases (i)-(ii) presented above when $f(i,x,\gamma) = \grad \check f(i+x,\bg) - \check f(i+x,\bg) (\grad h(i+x,\bg) + \grad q(x_i)$ and $\check f (i+x,\gamma)= (z+ \grad S_2(i,x,\gamma)) \psi(i+x,\gamma)$ (with $z$ the $d$-dimensional 0 or 1 vector). We intrinsically assume in the following that in the latter case, $\psi$ is chosen such that given $U=u$, ${\check f} \in \Fcal_\Rcal(H, \mathbb{Q}, \Lcal_u)$. 

For any function or stochastic function of $\theta$, $\ell(\theta)$, we denote by $\ell^{(1)}(\theta)$ and $\ell^{(2)}(\theta)$ its gradient and Hessian respectively. Minimizing~\eqref{eq:TF} is equivalent to finding the zeros of its gradient
\begin{align*}
\mathrm{TF}_{n}^{(1)} (  \bff, \bG ; \theta) &= 2 \sum_{l=1}^p  \mathrm{DLR}^{(1)}_{n} (f_l,\bG;\theta) \times \mathrm{DLR}_{n}(f_l,\bG;\theta)
\end{align*}

\subsection{Asymptotic results}

In this section, we investigate the properties of the Takacs-Fiksel estimator minimizing~\eqref{eq:TF}. We let \ref{M:model}-\ref{M:ident} denote the set of assumptions on the model and test functions $f_j$:
\begin{enumerate}[($\mathcal M$1)] 
\item We assume that the perturbed Gibbs lattice model satisfies \ref{H:stability}-\ref{H:invariant}, \ref{H:range}-\ref{Q:momentexp}[\!1], is parameterized by~\eqref{eq:explambda} where $\theta \in \Theta$, an open bounded convex set of $\R^p$, with interior containing $\theta^\star$ the true parameter vector to be estimated, and is observed  in the sequence of increasing cubes $(W_n)_n$ with volume $|W_n|\to \infty$ as $n\to \infty$. We assume that the sequence of real numbers $(m_n)$ defined in~\eqref{eq:bn} is such that $m_n = \beta \log|W_n|$ for $\beta>1/2$ such that $\sum_{n\ge 1}|W_n|^{-\beta}<\infty$.
\label{M:model}
\item Let $L\ge p$. For any  $l=1,\dots,L$, $u\in \Dcal$, $i \in \Lcal_u$, $x\in \R^d$ and $\bg \in \Ccal$, the functions $f_l(i,x,\bg; \cdot)$ and $\Lambda(i,x,\bg; \cdot)$ are twice continuously differentiable on $\Theta$. For $k=0,1,2$, we assume that each component $f_l^{(k)}(i,x,\bg;\theta)$ can be decomposed as $\phi_1(x;\theta) \phi_2(i+x,\bg;\theta)$, $\phi_1,\phi_2$ being different for different components of $f_l^{(k)}$, such that $\phi_2$ satisfies~\ref{TI}-\ref{FR}. \label{M:f}
\item For any $j,k=1,\dots,p$, $l=1,\dots,L$ and $\theta\in \Theta$, \ref{Af} is satisfied for  $f=f_l$, $(f^{(1)}_l)_j$ and $(f^{(2)}_l)_{jk}$, and that
\ref{Bf} is satisfied for $f=f_l,(f^{(1)}_l)_j \Lambda$, $f_l\Lambda^{(1)}_j/\Lambda \times$, $(f^{(2)}_l)_{jk}$, $(f_l^{(1)})_j(f_l^{(1)})_k$, $(f_l^{(1)})_j(\Lambda^{(1)})_j/\Lambda$ and $f_l (\Lambda^{(2)})_{jk}/\Lambda$. The assumptions \ref{Af}-\ref{Bf} correspond to:
\begin{enumerate}[(A){[$f$]}]
\item \label{Af} For any $u\in \R^d$ and $\theta \in \Theta$  %$i_1,\dots,i_{d+1} \in \Lcal_u$ 
\begin{align*}
\sup_{x \in \Qcal}& \; \Eu \left( |f(0,x,\bG_{0^c}; \theta)| \lambda_2(0,x,\bG_{0^c};\theta) \right) <\infty.
%& \; \Eu  
%\int_{\Qcal} \dots \int \prod_{p=1}^{d+1}
%\Big\{
%|f(i_p,x_p,\bG_{i_p^c};\theta)| \times \\
%&\qquad \qquad \Lambda(i_p,x_p, \bG_{\{\cup_{j=1}^{d+1} %i_j\}^c} \cup_{j=1}^{p-1} \{ i_j + x_j\} ;\theta^\star)
%\d x_p
%\Big\}<\infty
\end{align*}
\item \label{Bf} For any $u\in \R^d$ and $\theta\in \Theta$
%$i_1,\dots,i_{d+1} \in \Lcal_u$ 
\begin{align*}
&\sup_{x \in \Qcal}  \; \Eu \left( |f(0,x,\bG_{0^c}; \theta)| \frac{\Lambda(0,x,\bG_{0^c};\theta)}{\lambda_1(x;\theta^\star)} \right) <\infty\\
%& \Eu \int_{\Qcal} \dots \int  \prod_{p=1}^{d+1} \left\{ f(i_p,x_p,\bG_{i_p^c}; \theta) \Lambda(i_p,x_p,\bG_{i_p^c}; \theta) \d x_p \right\} <\infty\\
&\sup_{n\ge 1} \sup_{x,y \in \Qcal} \Eu \left( |f(0,x,\bG_{0^c}; \theta)| \;
\frac{
 \Lambda_n(0,x,\bG_{0^c};\theta) \Lambda(0,y,\bG_{0^c};\theta)}{ \lambda_1(x;\theta^\star) \lambda_1(y;\theta^\star)}
\right)
<\infty.
%\\
%&\sup_{n\ge 1} \Eu \int_{\Qcal} \dots \int_{\Qcal} \Big\{
%\prod_{p=1}^{d+1} |f(i_p,x_p,\bG_{i_p^c}; %\theta)|\Lambda(i_p,x_p,\bG_{i_p^c}; \theta)\times \\ 
%& \qquad \qquad \qquad \Lambda_n(i_p,y_p,\bG_{i_p^c}; \theta)
% \d x_p \d y_p\Big\} <\infty
\end{align*}
\end{enumerate}

%and $(\Lambda_n^{(1)}-\Lambda^{(1)})_j/(\Lambda_n-\Lambda)$.
%For any $\theta\in\Theta$, any $\theta^\prime\in\{\theta\} \cup \{\theta^\star\}$, %$k,k^\prime=0,1,2$ with $0<k+k^\prime\le 2$, and $l=1,\dots,L$
%\begin{align*}
%\E_u \bigg[ \int_{\Qcal}
%|f_j^{(k)}(0,x,\bG_{0^c};\theta)| \; \times \;  |\Lambda^{(k^\prime)}(0,x,\bG_{0^c};\theta)|
%\bigg] &<\infty\\
%\sup_{x \in \Qcal} \E_u\left[ 
%|f_l^{(k)}(0,x,\bG_{0^c};\theta)| \; \times \;  \frac{|\Lambda^{(k^\prime)}(0,x,\bG_{0^c};\theta^\prime)|}{
%\lambda_1(x)}  \right] &<\infty \\
%\int_{\Qcal} \E_u \left[  
%|f_l^{(k)}(0,x,\bG_{0^c};\theta)|^{2d}
%\; \times \;  |\Lambda^{(k^\prime)}(0,x,\bG_{0^c};\theta^\prime)|^{2d}\right]^{1/2d}\d x &<\infty.
%\end{align*} 
\label{M:ergodic}
\item We assume that for any $u \in \Dcal$, the quantity $\mathcal E_u$ defined by~\eqref{eq:finiteEu} is finite.\label{M:var} 
%and that 
%\begin{align*}
%\mathcal E_{\mathbb Q,u} = \sum_{j \in \Lcal_u} \P \left[ X-Y \in B(j,R)\right] <\infty
%|\{B(X-Y^\prime,R) \cup B(X^\prime - Y,R) \cup B(X-Y,R) \}  \cap (W_n\ominus R)_{-\check Y}
%\right] < \infty
%\end{align*}
%where $X,Y$ are independent random variables with distribution $\mathbb Q$.  
\item $\lambda_{\min}(\mathcal I(\theta^\star))>0$ where $\mathcal I(\theta^\star) =\sum_{l=1}^L
\mathcal I_l(\theta^\star)\mathcal I_l(\theta^\star)^\top$ with
\begin{equation} \label{eq:Il}
\mathcal I_l(\theta^\star) = \E_u \left[
\int_{\Qcal} f_l(0,x,\bG_{0^c};\theta^\star)    \Lambda^{(1)}(0,x,\bG_{0^c};\theta^\star) \; \d x
\right]
\end{equation}
and where $\lambda_{\min}(M)$, for a symmetric squared matrix $M$, stands for the smallest eigenvalue of $M$.
\label{M:ident}
\end{enumerate}

Assumption~\ref{M:model} ensures in particular, from Theorem~\ref{thm:existence}, the existence of at least one perturbed lattice stationary Gibbs measure. The assumption on $W_n$ is for instance satisfied for the simple sequence of cubes $W_n=[-n^{1/d}/2,n^{1/d}/2]^d$ with volume $|W_n|=n$ and $\beta>1$. Domains other than cubes could be considered. Finally, a look at the proof of Theorem~\ref{thm:ergodic} in particular shows that the sequence $m_n$ is actually related to the tail of the distribution $\mathbb Q$. More precisely it corresponds to an upper-bound of $-\log\P(|X|>m_n)$ for $X\sim \mathbb Q$. Note that if $\mathbb Q$ is compactly supported, such that $\P(|X|\ge K)=0$, one may replace $m_n$ by $K$ in the methodology. Assumption~\ref{M:f} requires specific form of the test functions $f_j$ and regularity conditions. For exponential models, such an assumption is satisfied for test functions leading to~\eqref{eq:LPL} and~\eqref{eq:evariational}.
Assumption~\ref{M:ergodic} is essentially used in Theorem~\ref{thm:ergodic} to prove ergodic results for empirical versions of DLR-type equations. Assumption~\ref{M:var} is used to control the conditional variance of $\widetilde{\mathrm{DLR}}_n(f_l; \theta^\star)$. Finally, Assumption~\ref{M:ident} corresponds to an identifiability condition.
%for test functions $f_l$ satisfying in particular~\ref{M:f}. 
%The condition $\mathcal E_u<\infty$ \JF{TODO}. 
%The assumption 
% $\mathcal E_{\mathbb Q,u}$ places few restrictions on the distribution $\mathbb Q$. It is obviously satisfied for any distribution with compact support and also satisfied for the Gaussian and exponential distributions described in Section~\ref{sec:hamiltonian}. 

We start with an ergodic theorem for DLR-type equations. Before that, for $A_n$ a random variable, $a_n$ a sequence of positive real numbers and $u \in \Dcal$, we write $A_n=O_{\P_u}(a_n)$ if the following property is satisfied: for any $\varepsilon=\varepsilon(u)$, there exists some finite $M=M(u)>0$ such that for $n$ large enough
\begin{equation}\label{eq:OPu}
    \mathrm P \left( |A_n| \ge M a_n \,\mid U=u \, \right) \le \varepsilon.
\end{equation}
It is worth mentioning that if $\Eu|A_n|^p<\infty$ for some $p>0$, by conditional Markov's inequality, we have that $A_n = O_{\P_u} ( (\Eu|A_n|^p)^{1/p})$.
\begin{theorem}\label{thm:ergodic} 
Assume~\ref{M:model} and let $f$ be a function satisfying \ref{M:f} and \ref{M:ergodic}. We define the random variables 
\begin{align*}
A_{W_n}(f,\bG;\theta) &= \sum_{i\in \Lcal_U} b_n(i,X_i) f(i,X_i,\bG_{i^c};\theta)\\
B_{W_n}(f,\bG;\theta) &= \sum_{i\in \Lcal_U} \int_{\Qcal} b_n(i,x) f(i,x,\bG_{i^c};\theta) \Lambda_n(i,x,\bG_{i^c};\theta) \d x\\
\widetilde B_{W_n}(f,\bG;\theta) &= \sum_{i\in \Lcal_U} \int_{\Qcal} b_n(i,x) f(i,x,\bG_{i^c};\theta) \Lambda(i,x,\bG_{i^c};\theta) \d x\\
\mathrm{DLR}_n(f,\bG;\theta) &= A_{W_n}(f,\bG;\theta) -B_{W_n}(f,\bG;\theta)\\
\widetilde{\mathrm{DLR}}_n(f,\bG;\theta) &= A_{W_n}(f,\bG;\theta) -\widetilde B_{W_n}(f,\bG;\theta).
\end{align*}
Then,  as $n\to \infty$, given $U=u$, we have the following two statements.\\
(i)
\begin{align} 
&|W_n|^{-1} A_{W_n}(f, \bG;\theta) \stackrel{a.s.}{\longrightarrow}
\Eu \int_{\Qcal} f(0,x, \bG_{0^c}; \theta) \Lambda(0,x, \bG_{0^c}; \theta^\star) \d x \label{eq:limA}\\
&|W_n|^{-1} B_{W_n}(f, \bG;\theta) \stackrel{a.s.}{\longrightarrow}
\Eu \int_{\Qcal} f(0,x, \bG_{0^c}; \theta) \Lambda(0,x, \bG_{0^c}; \theta) \d x \label{eq:limB}\\
&|W_n|^{-1} \mathrm{DLR}_n(f, \bG;\theta) \stackrel{a.s.}{\longrightarrow}
\Eu \int_{\Qcal} f(0,x, \bG_{0^c}; \theta) \left\{
\Lambda(0,x, \bG_{0^c}; \theta^\star) - \Lambda(0,x, \bG_{0^c}; \theta)
\right\} \d x\label{eq:limDLR} \\
&|W_n|^{-1} \left( \mathrm{DLR}_n(f , \bG;\theta) -  \widetilde{\mathrm{DLR}}_n(f , \bG;\theta) \right) \stackrel{a.s.}{\longrightarrow} 0. \label{eq:DLRDLRtilde}
\end{align}
(ii) 
\begin{equation}
    \label{DLRDLRtildeprobability}
\mathrm{DLR}_n(f , \bG;\theta) -  \widetilde{\mathrm{DLR}}_n(f , \bG;\theta) = O_{\P_u} \Big(  |W_n|^{1-\beta}\Big)
\end{equation}
where $\beta$ is defined in~\ref{M:model}.
\end{theorem}
The proof of this result is postponed to~\ref{app:ergodic}. The next result establishes the consistency of the Takacs-Fiksel estimator.

\begin{theorem}\label{thm:convTF} Under the assumptions~\ref{M:model}-\ref{M:ident}, we have the following two statements obtained as $n\to \infty$.\\
(i) Let $v_n= \sup_l\Var_u[ \widetilde{\mathrm{DLR}}_n(f_l,\bG;\theta^\star)]$, where $\Var_u[ \widetilde{\mathrm{DLR}}_n(f_l,\bG;\theta^\star)]$ stands for the conditional variance of $\widetilde{\mathrm{DLR}}_n(f_l,\bG;\theta^\star)$ given $U=u$. Then,  $v_n= O(|W_n|)$.\\

%Assume in addition that  with \JF{??} then 
%\begin{equation}\label{eq:assumptionvn}
%    \JF{??}{|W_n|^{1-2/d} \log(|W_n|)^{2+\varepsilon}} = o(\sqrt{v_n}).
%\end{equation} 
(ii) The sequence of local maximizers $\hat \theta$ satisfies
\begin{equation*}
\hat \theta - \theta^\star = O_{\P_u} \left(v_n^{1/2}|W_n|^{-1}\right) =O_{\P_u} \left(|W_n|^{-1/2} \right).    
\end{equation*}
\end{theorem}

Let us comment on the rate of convergence obtained in~(ii). As mentioned in the introduction, standard continuous Gibbs point processes are non hyperuniform and results regarding the Takacs-Fiksel estimator for these point processes, see~\citep{coeurjolly2012takacs}, show that the rate of convergence is $|W_n|^{-1/2}$. As a comparison, we obtain an upper-bound for the rate of convergence with similar rate. However, we also show that this rate can be reduced if one proves that $v_n=o(|W_n|)$. If $\bG$ is hyperuniform, it can be deduced (immediate if $|W_n|$ is an increasing ball) that $\Var[ N(W_n) ]=o(|W_n|)$. However, $N(W_n)$ is only the left-hand side of one particular $\mathrm{DLR}$ test function with $f=1$. It is an open and probably complex question, to derive conditions on a test function, say $f_l$, under which the hyperuniformity property of $\bG$ implies that $\Var_u[ \widetilde{\mathrm{DLR}}_n(f_l,\bG;\theta^\star)]= o(|W_n|)$. This question, of great importance from a statistical point of view, is left for future research.

\section{Simulation and numerical study} \label{sec:simulation}

In this section, we present simulation results and evaluate the performance of the estimator based on the Takacs-Fiksel method. The simulated point patterns were generated using a Metropolis-Hastings algorithm with a Gaussian move proposal and a Strauss interaction. The point patterns are generated on  $[-30,30]^2$ and then clipped to the observation domain $[-\ell,\ell]^2$ (for $\ell=8,12,16$). Replicated point patterns are generated from the same seed, for different sets of parameters $\theta_1 := \frac{1}{\sigma^2}$, where $\sigma$ is the standard deviation of the Gaussian move, $\theta_2$ is the strength of the Strauss interaction, and $R$ is the range of the Strauss interaction.

\begin{figure}[H]
\centering
\begin{subfigure}[b]{0.43\textwidth}
\centering
\includegraphics[width=1\textwidth,trim={.8cm .8cm .8cm .8cm},clip]{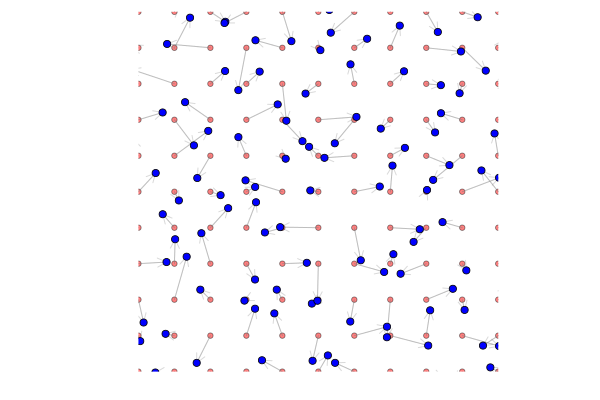}
\caption{$\theta_1 = 4$, $\theta_2 = -\log 0.9 $, $R= 0.5$}
\end{subfigure}
\begin{subfigure}[b]{0.43\textwidth}
\centering
\includegraphics[width=1\textwidth,trim={.8cm .8cm .8cm .8cm},clip]{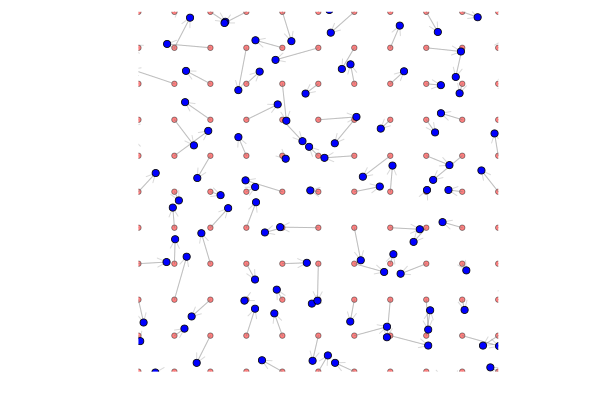}
\caption{$\theta_1 = 4$, $\theta_2 = -\log 0.9$, $R = 1.5$}
\end{subfigure}
    
\begin{subfigure}[b]{0.43\textwidth}
\centering
\includegraphics[width=1\textwidth,trim={.8cm .8cm .8cm .8cm},clip]{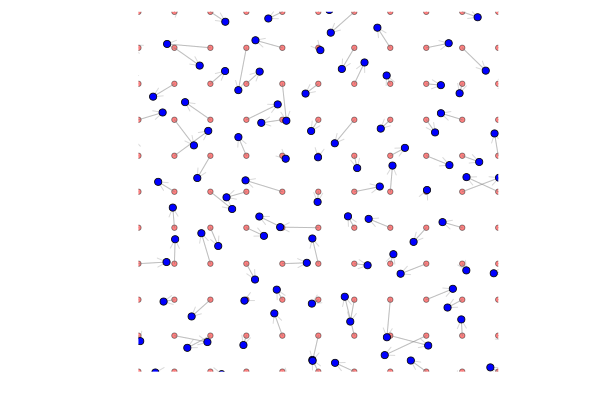}
\caption{$\theta_1 = 4$, $\theta_2 = -\log 0.1$, $R = 0.5$}
\end{subfigure}
\begin{subfigure}[b]{0.43\textwidth}
\centering
\includegraphics[width=1\textwidth,trim={.8cm .8cm .8cm .8cm},clip]{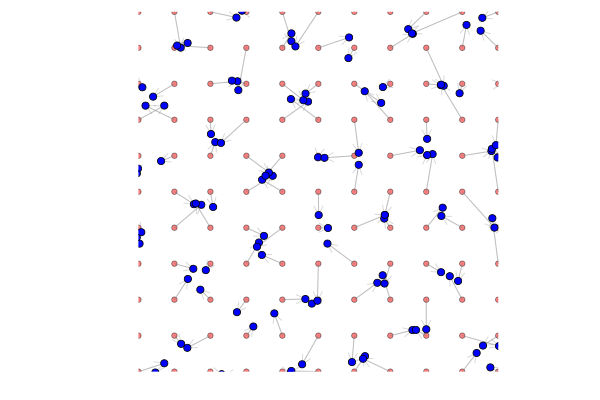}
\caption{$\theta_1 = 4$, $\theta_2 = -\log 0.1$, $R = 1.5$}
\end{subfigure}
    
\begin{subfigure}[b]{0.43\textwidth}
\centering
\includegraphics[width=1\textwidth,trim={.8cm .8cm .8cm .8cm},clip]{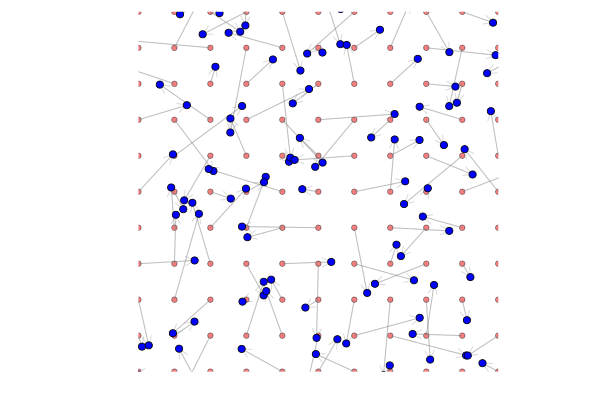}
\caption{$\theta_1 = 1$, $\theta_2 = -\log 0.9$, $R = 0.5$}
\end{subfigure}
\begin{subfigure}[b]{0.43\textwidth}
\centering
\includegraphics[width=1\textwidth,trim={.8cm .8cm .8cm .8cm},clip]{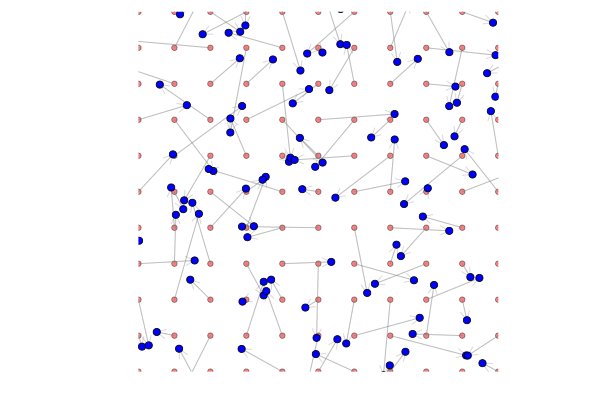}
\caption{$\theta_1 = 1$, $\theta_2 = -\log 0.9$, $R = 1.5$}
\end{subfigure}
    
\begin{subfigure}[b]{0.43\textwidth}
\centering
\includegraphics[width=1\textwidth,trim={.8cm .8cm .8cm .8cm},clip]{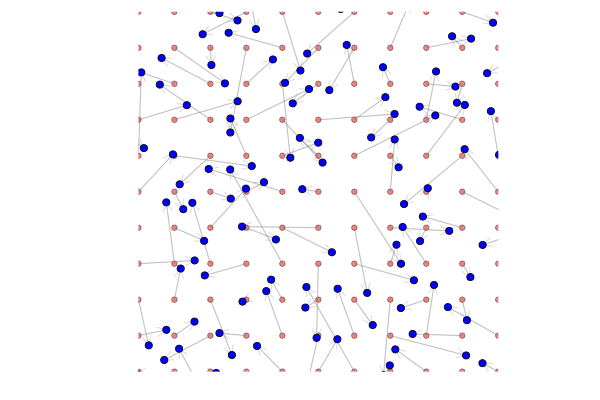}
\caption{$\theta_1 = 1$, $\theta_2 = -\log 0.1$, $R = 0.5$}
\end{subfigure}
\begin{subfigure}[b]{0.43\textwidth}
\centering
\includegraphics[width=1\textwidth,trim={.8cm .8cm .8cm .8cm},clip]{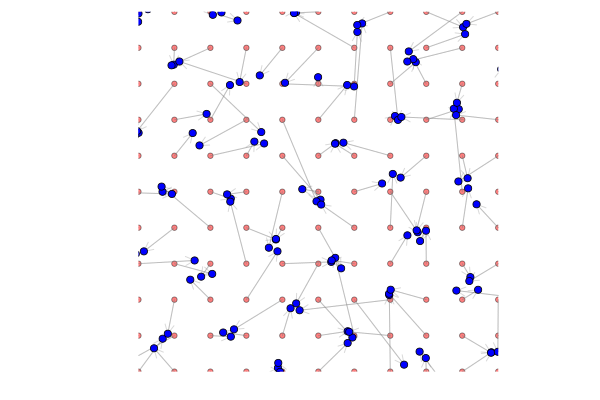}
\caption{$\theta_1 = 1$, $\theta_2 = -\log 0.1$, $R = 1.5$}
\end{subfigure}
\caption{Simulated Gibbs perturbed lattices with Strauss interaction and Gaussian moves.} \label{fig:sim}
\end{figure}

In these simulations, we observe an interesting phenomenon regarding the effect of $\theta_2$ and $R$ on clustering behavior. For regular continuous Gibbs point processes with  Strauss interaction, the larger the value of $\theta_2$, the more repulsive the interaction, and similarly, the larger $R$ is, the more repulsive the point process. For a fixed value of $\theta_1$, we observe that when $\theta_2$ is small (i.e., weak interaction), the influence of $R$ is negligible and the resulting point patterns are nearly indistinguishable. In this regime, the configurations display clusters of points separated by large empty regions, reminiscent of a Poisson point process. In contrast, when $\theta_2$ is large (i.e., strong interaction), the behavior changes markedly depending on the value of $R$. For $R = 0.5$, the configuration exhibits repulsive behavior: the points are more evenly spaced, with only a few pairs located close to one another. Conversely, for $R = 1.5$, a clear clustering phenomenon emerges, with groups of three to four interacting points occurring more frequently. Interestingly, these clusters themselves appear to be regularly spaced, as though they repel each other. According to Proposition \ref{prop:HU}, at least one Gibbs perturbed lattice with Strauss interaction and Gaussian moves is a hyperuniform point process. Figure~\ref{fig:sim} shows that we can obtain hyperuniform processes with different clustering behaviors simply by tuning the model parameters.

\begin{table}[H]
\centering
\begin{tabular}{l|lllllllll}
\hline
\rowcolor{gray!40} $(\theta_1, \theta_2)$  & \multicolumn{9}{c}{$W_n=:[-\ell,\ell]^2$} \\
\rowcolor{gray!40}&\multicolumn{3}{c}{$\ell=8$} & \multicolumn{3}{c}{$\ell=12$} & \multicolumn{3}{c}{$\ell=16$} \\
  \rowcolor{gray!40}& mean & sd & {\small RMSE} & mean & sd & {\small RMSE} & mean & sd & {\small RMSE} \\
%\hline
\hline \hline
%\multicolumn{10}{l}{}\\
\rowcolor{gray!20}\multicolumn{10}{l}{$R=0.5$} \\
\hline
$\theta_1 = 4$ & 3.88 & 0.19 & 0.23 & 4.01 & 0.16 & 0.16 & 4.10 & 0.18 & 0.21 \\
$\theta_2 \approx 0.11$ & 0.11 & 0.19 & 0.19 & 0.03 & 0.09 & 0.12 & 0.12 & 0.11 & 0.11 \\
\hline
$\theta_1 = 1$ & 1.10 & 0.07 & 0.12 & 0.99 & 0.07 & 0.07 & 1.00 & 0.05 & 0.05 \\
$\theta_2  \approx 0.11$ & 0.15 & 0.09 & 0.10 & 0.07 & 0.11 & 0.12 & 0.11 & 0.06 & 0.06 \\
\hline
$\theta_1 = 4$ & 4.19 & 0.41 & 0.45 & 4.04 & 0.20 & 0.20 & 4.05 & 0.19 & 0.19 \\
$\theta_2 \approx0.69$ & 0.84 & 0.28 & 0.32 & 0.69 & 0.17 & 0.17 & 0.71 & 0.08 & 0.08 \\
\hline
$\theta_1 = 1$ & 1.01 & 0.08 & 0.08 & 0.98 & 0.04 & 0.05 & 1.02 & 0.03 & 0.04 \\
$\theta_2 \approx0.69$ & 0.71 & 0.16 & 0.16 & 0.65 & 0.17 & 0.17 & 0.68 & 0.15 & 0.15 \\
\hline
$\theta_1 = 4$ & 3.98 & 0.37 & 0.37 & 3.93 & 0.15 & 0.17 & 4.01 & 0.15 & 0.15 \\
$\theta_2 \approx2.3$ & 2.16 & 0.30 & 0.33 & 2.33 & 0.20 & 0.21 & 2.50 & 0.25 & 0.32 \\
\hline
$\theta_1 = 1$ & 0.99 & 0.07 & 0.07 & 1.01 & 0.05 & 0.05 & 0.98 & 0.03 & 0.04 \\
$\theta_2 \approx2.3$ & 2.30 & 0.52 & 0.52 & 2.26 & 0.23 & 0.23 & 2.24 & 0.18 & 0.19 \\
\hline \hline
%\multicolumn{10}{l}{}\\
\rowcolor{gray!20}\multicolumn{10}{l}{$R=1.5$} \\
\hline
$\theta_1 = 4$ & 4.03 & 0.50 & 0.50 & 4.16 & 0.30 & 0.34 & 4.06 & 0.27 & 0.27 \\
$\theta_2 \approx 0.11$ & 0.11 & 0.06 & 0.07 & 0.11 & 0.07 & 0.07 & 0.13 & 0.05 & 0.06 \\
\hline
$\theta_1 = 1$ & 1.06 & 0.10 & 0.12 & 1.05 & 0.07 & 0.09 & 0.99 & 0.03 & 0.03 \\
$\theta_2 \approx 0.11$ & 0.11 & 0.09 & 0.09 & 0.10 & 0.06 & 0.06 & 0.11 & 0.04 & 0.04 \\
\hline
$\theta_1 = 4$ & 4.16 & 0.27 & 0.31 & 4.07 & 0.31 & 0.32 & 4.00 & 0.18 & 0.18 \\
$\theta_2  \approx0.69$ & 0.77 & 0.09 & 0.11 & 0.76 & 0.05 & 0.08 & 0.69 & 0.05 & 0.05 \\
\hline
$\theta_1 = 1$ & 1.09 & 0.09 & 0.12 & 1.01 & 0.05 & 0.05 & 0.99 & 0.06 & 0.06 \\
$\theta_2 \approx0.69$ & 0.65 & 0.10 & 0.11 & 0.70 & 0.07 & 0.08 & 0.74 & 0.03 & 0.06 \\
\hline
$\theta_1 = 4$ & 92.97 & 49.08 & 101.61 & 6.27 & 0.47 & 2.32 & 5.33 & 0.61 & 1.46 \\
$\theta_2 \approx2.3$ & 101.57 & 55.77 & 113.86 & 3.21 & 0.46 & 1.01 & 3.43 & 0.37 & 1.19 \\
\hline
$\theta_1 = 1$ & 1.21 & 0.12 & 0.24 & 1.02 & 0.07 & 0.07 & 1.10 & 0.02 & 0.10 \\
$\theta_2 \approx2.3$ & 2.60 & 0.36 & 0.46 & 2.38 & 0.22 & 0.24 & 2.43 & 0.14 & 0.19 \\
\hline
\end{tabular}
\caption{Mean, standard deviation and root-mean squared error (RMSE) of estimates based on 200 simulations from Gibbs perturbed lattices with Strauss interaction and Gaussian moves, and observed on $[-\ell,\ell]^2$ for $\ell=8,12,16$.}
	\label{table}
\end{table}

For each parameter configuration, 200 point patterns are simulated, and the parameters are estimated for each realization using the Takacs–Fiksel procedure with the edge corrections described in Section~\ref{sec:inference}. Due to edge correction effects, when $W_n=[-\ell,\ell]^2$ with $\ell=8,12,16$, the estimator uses on average 100, 256, and 576 points, respectively. As shown in Table~\ref{table}, the Takacs-Fiksel estimator performs excellently well. Both estimates of $\theta_1$ and $\theta_2$ exhibit very small biases, with standard deviations decreasing as $\ell$ increases. Consequently, the RMSEs also decrease with $\ell$ and remain low even for the smallest observation domain $[-8,8]^2$. One notable exception occurs for the parameters $\theta_1=4$, $\theta_2=-\log 0.1 \approx 2.3$, and $R=1.5$, corresponding to the clustering regime described earlier with a more constrained move. In this regime, the estimator performs poorly: enlarging the observation window improves performance, but it still fails to reach the estimation accuracy observed in other parameter regimes. The system is highly rigid, with strong repulsion between particles and a strong attraction to the lattice sites. This competition between antagonistic forces may explain the failure of the estimation procedure. Further tests on larger observation windows are necessary, but they are computationally expensive due to the cost of generating point patterns and performing the optimisation. For weaker interactions (i.e., $\theta_2=-\log 0.9 \approx 0.11$), the estimation procedure is consistently reliable, regardless of $\theta_1$ and $R$. In most parameter regimes, the RMSEs decay at a rate consistent with $|W_n|^{-1/2}$.

%%%%%%%%%%%%%%%%%%%%%%%%%%%%%%%%%%%%%%%%%%%%%%%%%%%%%%%%%%%%%%%%%%%%%%%
%%%%%%%%%%%%%%%%%%.  Appendix              %%%%%%%%%%%%%%%%%%%%%%%%%%%%
%%%%%%%%%%%%%%%%%%%%%%%%%%%%%%%%%%%%%%%%%%%%%%%%%%%%%%%%%%%%%%%%%%%%%%%

\begin{appendix}

\section{Auxiliary results for the proof of Theorem~\ref{thm:existence}} \label{app:aux.existence}

In this section, we prove auxiliary results needed for the proof of Theorem \ref{thm:existence}, which is that under assumptions \ref{H:stability}-\ref{H:invariant} and \ref{H:range}-\ref{Q:momentexp}[\!1] or \ref{H:summable}-\ref{Q:bounded}, $P_\Lcal(H, \mathbb{Q}) \neq \emptyset$. Let $\Lambda_n = \mathcal{L} \cap [-n,n]^d$, we define the empirical field $\overline{\P}_n$ as the probability measure $$\overline{\P}_n := \frac{1}{|\Lambda_n|} \sum_{i \in \Lambda_n} \widehat{\P}_n( \cdot - i)$$ where $\widehat{\P}_n = \bigotimes_{k \in \mathcal{L}} \P_{\Lambda_n + 2kn/|k|}$. The empirical field is invariant under translation in $\mathcal{L}$. Our aim is to prove that the empirical field exhibits an accumulation point under the local topology. The local topology is defined as the coarsest topology such that for any cylinder event, also called local event, $A$ the evaluation map $P \mapsto P(A)$ is continuous. The main tightness tool is the specific entropy. Let $P$ be an invariant under translation probability measure on $\Omega$,
\begin{equation}
    I_s(P) = \lim_{\Lambda \to \mathcal{L}} \frac{1}{|\Lambda|} I(P_{|\Lambda}| \mathbb{Q}^{\otimes \Lambda})
\end{equation}
where $I(\cdot|\cdot)$ is the Kullback-Leibler divergence and this limit exists due to a sub-additivity argument. The following proposition gives the properties of the specific entropy and the arguments that give the tightness of probability measures of lattice systems. 

\begin{proposition}[Theorem 15.12 and Proposition 15.14 in \citep{georgii2011book}] \label{prop: propriete_entropie_specifique}${  }$\\
%\begin{itemize}
 %   \item 
 For any probability measure $P$ on $\Omega$ invariant under any translation on $\mathcal{L}$ we define the specific entropy of $P$ as, 
\begin{equation}
    I_s(P) = \sup_{\Lambda \subset \mathcal{L}, |\Lambda| < \infty} \frac{1}{|\Lambda|} I(P_{|\Lambda} | Q^{\otimes \Lambda} ).
\end{equation}
Therefore, the specific entropy does not depend on the sequence of finite sub-lattices $(\Lambda_n)_{n \in \mathbb{N} }$. Moreover, the specific entropy, $I_s$, is affine and the level sets $\{I_s(\cdot) \leq C \}$ are compact and sequentially compact for the local topology.
%\end{itemize}
\end{proposition}

Now we prove that the specific entropy for the empirical fields is bounded in order to exhibit a candidate for a Gibbs measure on the lattice.  
\begin{proposition}
The empirical field $\left( \overline{\P}_n \right)_{n \in \N}$ has an accumulation point in the local topology that we denote by $\P_\infty$.
\end{proposition}

\begin{proof}
Since the specific entropy is affine, we have that
\begin{equation*}
    I_s(\overline{\P}_n) = \frac{1}{|\Lambda_n|} \sum_{i \in \Lambda_n } I_s(\widehat{\P}_n(\cdot - i)).
\end{equation*}
Furthermore, since the specific entropy does not depend on the sequence of sub-lattices, then for $i \in \Lambda_n$ we choose the sequence $\left(\Lambda_{(2k-1)n}+i\right)_{k \geq 1}$ and therefore
\begin{align*}
    I_s(\widehat{\P}_n( \cdot - i)) &= \lim_{k \to \infty} \frac{1}{(2k-1)^d|\Lambda_{n}|} I(\widehat{\P}_{n |\Lambda_{(2k-1)n}} | \mathbb{Q}^{\otimes \Lambda_{(2k-1)n}}) \\
    &= \lim_{k \to \infty} \frac{1}{(2k-1)^d|\Lambda_{n}|} (2k-1)^d I(\P_{n} | \mathbb{Q}^{\otimes \Lambda_{n}}) \\
    &= \frac{1}{|\Lambda_n|} I(\P_n | \mathbb{Q}^{\otimes \Lambda_{n}}).
\end{align*}
Combining everything we have
\begin{equation*}
    I_s(\overline{\P}_n) = \frac{1}{|\Lambda_n|} I(\P_n | \mathbb{Q}^{\otimes \Lambda_{n}}).
\end{equation*}
Finally, according to \ref{H:stability} and \ref{H:nondegenerate}, for $\bX \sim \P_n$, we have
\begin{align*}
    I(\P_n | \mathbb{Q}^{\otimes \Lambda_{n}}) &= - \log Z_{\Lambda_n} - \E(\tilde{H}(\bX)) \\
    & \leq (A + B - \log \mathbb{Q}(B(0,\varepsilon)) )|\Lambda_n|.
\end{align*}
As a consequence, for all $n \in \N$
\begin{equation}\label{ineq: bounded specific entropy}
    I_s(\overline{\P}_n) \leq A + B - \log \mathbb{Q}(B(0,\varepsilon)),
\end{equation}
and thus, according to Proposition \ref{prop: propriete_entropie_specifique}, $(\overline{\P}_n)_{n\in \N}$ has an accumulation point. 
\end{proof}
A function  $f$ is local if there is $\Delta \subset \Lcal$ finite such that for any $\bx \in \Qcal^\Lcal$, $f(\bx) = f(\bx_\Delta)$. Moreover, $f$ is tame if $\E(e^{f(\tilde \bX_\Delta)})< \infty$ for $\tilde \bX_\Delta \sim \mathbb{Q}^{\otimes \Delta}$.
\begin{lemma}\label{lem: specific entropy bound}
Let $f$ be a local and tame function on $\Delta \subset \mathcal{L}$ finite, for any stationary probability measure $P$ on $\Omega$, $\bX \sim P$ and $\tilde \bX_\Delta \sim \mathbb{Q}^{\otimes \Delta}$ we have 
\begin{equation*}
    \E[f(\bX)] \leq I_s(P) |\Delta| + \log \E\left[\e^{f(\tilde{\bX}_\Delta)}\right].
\end{equation*}
\end{lemma}

\begin{proof}
    By Jensen's inequality we have that 
    \begin{equation}\label{eq: lem ineq specific entropy1}
        \E_P[f(\bX)] = \E[f(\bX_\Delta)] \leq I(P_{|\Delta} | \mathbb{Q}^{\otimes \Delta}) + \log \E\left[\e^{f(\tilde{\bX}_\Delta)}\right].
    \end{equation}
    On the other hand, we know that 
    \begin{equation*}
        I_s(P) = \sup_{\Lambda \subset \mathcal{L}, |\Lambda| < \infty} \frac{1}{|\Lambda|} I(P_{|\Lambda} | \mathbb{Q}^{\otimes \Lambda} ).
    \end{equation*}
    Consequently, we have
    \begin{equation}\label{eq: lem ineq specific entropy2}
        I(P_{|\Delta} | \mathbb{Q}^{\otimes \Delta}) \leq I_s(P) |\Delta|.
    \end{equation}
    By combining \eqref{eq: lem ineq specific entropy1} and \eqref{eq: lem ineq specific entropy2}, we obtain the desired inequality. 
\end{proof}

%\begin{definition}
%    An absolutely summable potential $\Phi$ is a family of potentials such that 
%    \begin{equation}
%        ||\tilde{H}_i|| :=  \sum_{A \subset \mathcal{L}, i\cap A} ||\tilde \Phi_A||_{\infty}
%    \end{equation}
%    is finite. 
%\end{definition}

%\begin{theorem}[Georgii Theorem 4.23]\label{thm:existence georgii}
%Suppose $(E, \Xi)$ is a standard Borel space, $\mathbb{Q}$ a finite measure on $E$ and if $\Phi$ is an absolutely summable potential $\mathcal{G}(\Phi)$ is non-empty and compact. 
%\end{theorem}

\section{Proof of Theorem~\ref{thm:existence}}
\label{app:existence}

\begin{proof} We split the proofs into two parts. The first two show that the measure $\P_\infty$ satisfies DLR equations under the two different sets of assumptions.\\

\noindent{\sc $\hookrightarrow$ \underline{Part A}. $\P_\infty$ satisfies DLR equations under the assumptions~\ref{H:stability}-\ref{H:invariant} and \ref{H:range}-\ref{Q:momentexp}[\!1].} Let $f$ be a bounded local function, $\Delta \subset  \mathcal{L}$ finite and $\epsilon >0$. We split the proof into 7 steps.

\paragraph{Step 1}
Since $f$ is a local bounded function by convergence in the local topology, we have for $n$ large enough
\begin{equation}\label{ineq: dlr step1}
    \left| \int f \d\P_\infty - \int f \d\overline{\P}_n \right| \leq \epsilon \norm{f}_\infty.
\end{equation}

\paragraph{Step 2} We define $B_k = [-k,k]^d \oplus \Delta$, $\Delta_k = \Lambda_k \cap \mathcal{L}$ and we consider the following events
\begin{align*}
    I_{\Delta, l} &:= \{ \exists i \in \Delta, i+x_i \in B_l^c \}, \\
    T_{\Delta, k,l,R} &= \{ \exists j \in \Delta_k^c, j+x_j \in B_{l+R}\}, \\
    E_{k,l,R,\Delta} &=  I_{\Delta, l} \cup T_{\Delta, k,l,R}.
\end{align*}
We have that
\begin{equation}\label{ineq: dlr step2}
    \sup_{n \in \N} \left| \int f \d\overline{\P}_n - \int f \1_{E_{k,l,R,\Delta}^c} \d\overline{\P}_n \right| \leq \norm{f}_\infty \epsilon.
\end{equation}
Using the fact that $(X_i)_{i \in \mathcal{L}}$ are identically distributed under $\overline{\P}_n$ for any $n\ge 1$ and Markov inequality
\begin{align*}
    \overline{\P}_n (E_{k,l,R,\Delta} ) &\leq \sum_{i\in \Delta} \overline{\P}_n(i+X_i \in B_l^c ) + \sum_{j \in \Delta_k} \overline{\P}_n (j+X_j \in B_{l+R}) \\
    & \leq |\Delta| \overline{\P}_n (|X_0| \geq l) + |\Lambda_{l+R}| \overline{\P}_n(|X_0| \geq k-l-R) \\
    & \leq \left(\frac{|\Delta|}{l} + \frac{|B_{l+R}|}{k-l-R} \right) \E(|X_0|).
\end{align*}
Furthermore, using Lemma \ref{lem: specific entropy bound} and assumption \ref{Q:momentexp}[\!1] and \eqref{ineq: bounded specific entropy}, there exists $c>0$ such that for any $n \in \N$
\begin{equation*}
    \E[c |X_0|] \leq I_s(\overline{\P}_n) + \log \E[\e^{c|\tilde X |}] < \infty
\end{equation*}
for $\tilde X \sim \mathbb Q$. Consequently, there exists $C>0$ such that 
\begin{equation}\label{ineq: proba event interaction site distant}
    \overline{\P}_n (E_{k,l,R,\Delta} ) \leq C \left(\frac{|\Delta|}{l} + \frac{|B_{l+R}|}{k-l-R} \right) .
\end{equation}
As a consequence for $k$ and $l$ large enough we obtain \eqref{ineq: dlr step2}.

\paragraph{Step 3} For $n$ large enough we have,
\begin{equation}\label{ineq: dlr step3}
    \left| \int f \1_{E_{k,l,R, \Delta}^c} \d\overline{\P}_n - \int \tilde{f}_{\Delta,l} \1_{T_{k,l,R,\Delta}^c} \d\overline{\P}_n \right| \leq 2 \epsilon \norm{f}_\infty ,
\end{equation}
where
\begin{equation*}
    \tilde{f}_{\Delta,l}(\bx_{\Delta^c}) = \frac{1}{Z_\Delta(\bx_{\Delta^c})}\int f(\bx_\Delta' \cup \bx_{\Delta^c}) \1_{I_{l,\Delta}^c}(\bx_\Delta') \e^{-\tilde H_\Delta(\bx_\Delta' \cup \bx_{\Delta^c})} \mathbb{Q}^{\otimes \Delta}.
\end{equation*}
We consider $n$ large enough such that $\Delta \subset \Lambda_n$ and we define $\Lambda_n^* := \{ i\in \Lambda_n, \Delta-i \subset \Lambda_n \} $. By definition, we have
\begin{equation*}
    \int f \1_{E_{k,l,R,\Delta}^c} \d\overline{\P}_n = \frac{1}{|\Lambda_n|} \sum_{i\in\Lambda_n} \int f \1_{E_{k,l,R,\Delta}^c}(\cdot + i) \d\widehat{\P}_n,
\end{equation*}
and as a consequence we have 
\begin{align*}
    \left| \int f \1_{E_{k,l,R,\Delta}^c} \d\overline{\P}_n - \frac{1}{|\Lambda_n|} \sum_{i\in\Lambda_n^*} \int f \1_{E_{k,l,R,\Delta}^c}(\cdot + i) \d\widehat{\P}_n\right| \leq \frac{|\Lambda_n \setminus \Lambda_n^*|}{|\Lambda_n|} \norm{f}_\infty.
\end{align*}
Let us consider
\begin{align*}
    \tilde{f}_{\Delta, E^c} (\bx) = \frac{1}{Z_\Delta(\bx_{\Delta^c})} \int f \1_{E_{k,l,R,\Delta}^c}(\bx_{\Delta^c} \cup \bx_{\Delta}') \e^{-H_\Delta(\bx_{\Delta^c} \cup \bx_{\Delta}')} \mathbb{Q}^{\otimes \Delta}
\end{align*}
For $i \in \Lambda_n^*$, we have
\begin{align*}
    \int &\tilde{f}_{\Delta, E^c}( \cdot + i) \d\widehat{\P}_n = \! \iint f \1_{E_{k,l,R,\Delta-i}^c} \left( \bx_{(\Delta-i)^c} \cup \bx_{\Delta-i}' \right) \frac{\e^{-\tilde H_{\Delta-i}\left( \bx_{(\Delta-i)^c} \cup \bx_{\Delta-i}'  \right)}}{Z_{\Delta-i}\left( \bx_{(\Delta-i)^c}\right)}  \mathbb{Q}^{\otimes (\Delta -i)} \d\widehat{\P}_n.
\end{align*}
%\begin{align*}
%    \int &\tilde{f}_{\Delta, E^c}( \cdot + i) \d\widehat{\P}_n \\
%    =& \int \int f \1_{E_{k,l,R,\Delta}^c}(\tau_i(x)_{\Delta^c} \cup x_{\Delta}') \frac{\e^{-\tilde H_\Delta(\tau_i(x)_{\Delta^c} \cup x_{\Delta}')}}{Z_\Delta(\tau_i(x)_{\Delta^c})}  \mathbb{Q}^{\otimes \Delta} \d\widehat{\P}_n \\
%    =& \int \int f \1_{E_{k,l,R,\Delta}^c}\left(\tau_i \left( x_{\tau_{-i}(\Delta)^c} \cup x_{\tau_{-i}(\Delta)}' \right) \right) \frac{\e^{-\tilde H_\Delta\left(\tau_i \left( x_{\tau_{-i}(\Delta)^c} \cup x_{\tau_{-i}(\Delta)}' \right) \right)}}{Z_\Delta\left(\tau_i\left( x_{\tau_{-i}(\Delta)^c} \right)\right)}  \mathbb{Q}^{\otimes \Delta} \d\widehat{\P}_n \\
%    =& \int \int f \1_{E_{k,l,R,\Delta}^c}\left(\tau_i \left( x_{\tau_{-i}(\Delta)^c} \cup x_{\tau_{-i}(\Delta)}' \right) \right) \frac{\e^{-\tilde H_{\tau_{-i}(\Delta)}\left( x_{\tau_{-i}(\Delta)^c} \cup x_{\tau_{-i}(\Delta)}'  \right)}}{Z_{\tau_{-i}(\Delta)}\left( x_{\tau_{-i}(\Delta)^c}\right)}  \mathbb{Q}^{\otimes \tau_{-i}(\Delta)} \d\widehat{\P}_n.
%\end{align*}
Since $\Delta-i \subset \Lambda_n$ and that $\widehat{\P}_n$ satisfies the DLR equations for $\Delta+i$, we have 
\begin{align*}
    \int \tilde{f}_{\Delta, E^c}(\cdot +i) \d\widehat{\P}_n = \int f \1_{E_{k,l,R,\Delta}^c}(\cdot +i) \d\widehat{\P}_n
\end{align*}
Similarly we have
\begin{equation*}
    \left| \int \tilde{f}_{\Delta, E^c} \d\overline{\P}_n - \frac{1}{|\Lambda_n|} \sum_{i\in\Lambda_n^*} \int \tilde{f}_{\Delta, E^c}(\cdot +i) \d\widehat{\P}_n\right| \leq \frac{|\Lambda_n \setminus \Lambda_n^*|}{|\Lambda_n|} \norm{f}_\infty.
\end{equation*}
Finally, we have that $\1_{E_{k,l,R,\Delta}^c} = \1_{I_{l,\Delta}^c} \1_{T_{k,l,R,\Delta}^c}$ and since $T_{k,l,R,\Delta}^c$ is an event that depends only on the configuration outside $\Delta$ we have
\begin{equation*}
    \tilde{f}_{\Delta, E^c} = \tilde{f}_{\Delta, l} \1_{T_{k,l,R,\Delta}^c}.
\end{equation*}
Therefore, by combining everything above we obtain \eqref{ineq: dlr step3} for $n$ large enough. 

\paragraph{Step 4}
In this step we consider the following 
\begin{equation*}
    \widehat{f}_{\Delta, l}(\mathbf x) = \frac{1}{Z_{\Delta,l}(\bx_{\Delta^c})} \int f(\bx_\Delta' \cup \bx_{\Delta^c}) \1_{I_{l,\Delta}^c}(\bx_\Delta') \e^{-\tilde H_\Delta(\bx_\Delta' \cup \bx_{\Delta^c})} \mathbb{Q}^{\otimes \Delta}
\end{equation*}
where
\begin{equation*}
    Z_{\Delta,l}(\bx_{\Delta^c}) = \int \1_{I_{l,\Delta}^c}(\bx_\Delta') \e^{-\tilde H_\Delta(\bx_\Delta' \cup \bx_{\Delta^c})} \mathbb{Q}^{\otimes \Delta}. 
\end{equation*}
For $k$, $l$ and $n$ large enough we have,
\begin{equation}\label{ineq: dlr step4}
    \left| \int \tilde{f}_{\Delta,l} ( \bx_{\Delta^c}) \1_{T_{k,l,R,\Delta}^c} (\bx_{\Delta_k^c}) \d\overline{\P}_n - \int \widehat{f}_{\Delta,l} ( \bx_{\Delta^c}) \1_{T_{k,l,R,\Delta}^c} (\bx_{\Delta_k^c}) \d\overline{\P}_n\right| \leq 3\epsilon \norm{f}_\infty.
\end{equation}
Since the difference between the normalisation constants is
\begin{equation*}
    Z_{\Delta}(\bx_{\Delta^c}) - Z_{\Delta,l}(\bx_{\Delta^c}) =  \int \1_{I_{l,\Delta}}(\bx_\Delta') \e^{-\tilde H_\Delta(\bx_\Delta' \cup \bx_{\Delta^c})} \mathbb{Q}^{\otimes \Delta},
\end{equation*}
we have the following upper bound
\begin{align*}
\Big| \int \tilde{f}_{\Delta,l}  \1_{T_{k,l,R,\Delta}^c}  \d\overline{\P}_n - &\int \widehat{f}_{\Delta,l} \1_{T_{k,l,R,\Delta}^c}  \d\overline{\P}_n\Big|  \leq \norm{f}_\infty \times \\
&\hspace*{2cm}\int \frac{1}{Z_\Delta(\bx_{\Delta^c})} \int  \1_{I_{l,\Delta}}(\bx_\Delta') \e^{-\tilde H_\Delta(\bx_\Delta' \cup \bx_{\Delta^c})} \mathbb{Q}^{\otimes \Delta} \d\overline{\P}_n
\end{align*}
Furthermore, for $n$ large enough we can apply \eqref{ineq: dlr step3} for $f \equiv 1$, and obtain that
\begin{equation*}
    \int \frac{1}{Z_\Delta(\bx_{\Delta^c})} \int  \1_{I_{l,\Delta}}(\bx_\Delta') \e^{-\tilde H_\Delta(\bx_\Delta' \cup \bx_{\Delta^c})} \mathbb{Q}^{\otimes \Delta} \d\overline{\P}_n \leq 2 \epsilon + \overline{\P}_n (E_{k,l,R,\Delta}),
\end{equation*}
and by \eqref{ineq: proba event interaction site distant} we know that for $k, l$ large enough we have,
\begin{equation*}
    \sup_{n \in \N} \overline{\P}_n( E_{k,l,R, \Delta}) \leq \epsilon,
\end{equation*}
and therefore we have \eqref{ineq: dlr step4}.

\paragraph{Step 5}
For $k$, $l$ and $n$ large enough we have
\begin{equation}\label{ineq: dlr step5}
    \left| \int \widehat{f}_{\Delta,l} ( \bx_{\Delta^c}) \1_{T_{k,l,R,\Delta}^c} (\bx_{\Delta_k^c}) \d\overline{\P}_n - \int \widehat{f}_{\Delta,l} ( \bx_{\Delta^c}) \1_{T_{k,l,R,\Delta}^c} (\bx_{\Delta_k^c}) \d\P_\infty \right| \leq 3 \epsilon \norm{f}_\infty. 
\end{equation}
We can remark that under the event $T_{k,l,R,\Delta}^c$, $\widehat{f}_{\Delta,l}$ depends only on the configurations in $\Delta_k \setminus \Delta$ since the points originating from $\Delta_k^c$ are too far away to interact with the points from $\Delta$. Therefore, we have
\begin{align*}
    \Big| \int \widehat{f}_{\Delta,l}\1_{T_{k,l,R,\Delta}^c} \d \overline{\P}_n &- \int \widehat{f}_{\Delta,l} \1_{T_{k,l,R,\Delta}^c}  \d P_\infty \Big| \\
    \leq&  \left| \int \widehat{f}_{\Delta,l}\1_{T_{k,l,R,\Delta}^c} (\bx_{\Delta_k \setminus \Delta}) \d \overline{\P}_n - \int \widehat{f}_{\Delta,l}(\bx_{\Delta_k \setminus \Delta}) \d \overline{\P}_n \right| \\
    &+ \left| \int \widehat{f}_{\Delta,l}(\bx_{\Delta_k \setminus \Delta}) \d \overline{\P}_n - \int \widehat{f}_{\Delta,l}(\bx_{\Delta_k \setminus \Delta}) \d \P_\infty \right| \\
    &+ \left| \int \widehat{f}_{\Delta,l}(\bx_{\Delta_k \setminus \Delta}) \d \P_\infty - \int \widehat{f}_{\Delta,l} \1_{T_{k,l,R,\Delta}^c}  \d \P_\infty \right| \\
    \leq & \norm{f}_\infty \left( \overline{\P}_n( T_{k,l,R,\Delta}) + \P_\infty( T_{k,l,R,\Delta}) \right) \\
    & + \left| \int \widehat{f}_{\Delta,l}(\bx_{\Delta_k \setminus \Delta}) \d \overline{\P}_n - \int \widehat{f}_{\Delta,l}(\bx_{\Delta_k \setminus \Delta}) \d \P_\infty \right|.
\end{align*}
Using similar arguments as the ones used in Step 2, there exists a constant $C> 0$ such that
\begin{align*}
    \sup_{n\in\N} \overline{\P}_n (T_{k,l,R,\Delta}) \leq  \frac{C|B_{l+R}|}{k-l-R}.
\end{align*}
As a consequence for $k$ large enough we have 
\begin{equation*}
    \overline{\P}_n( T_{k,l,R,\Delta}) + \P_\infty( T_{k,l,R,\Delta}) \leq 2\epsilon.
\end{equation*}
Once $k$ is fixed, by convergence in the local topology we have for $n$ large enough
\begin{equation*}
    \left| \int \widehat{f}_{\Delta,l}(\bx_{\Delta_k \setminus \Delta}) \d \overline{\P}_n - \int \widehat{f}_{\Delta,l}(\bx_{\Delta_k \setminus \Delta}) \d \P_\infty \right| \leq \epsilon \norm{f}_\infty. 
\end{equation*}

\paragraph{Step 6}
For $l$ large enough we have
\begin{equation} \label{ineq: dlr step6}
    \left| \int \widehat{f}_{\Delta,l} \1_{T_{k,l,R,\Delta}^c} \d \P_\infty - \int \tilde{f}_{\Delta,l} \d \P_\infty \right| \leq 2\epsilon \norm{f}_\infty. 
\end{equation}
Following similar computations as in Step 4 we show that
\begin{equation*}
    \left| \int \widehat{f}_{\Delta,l} \1_{T_{k,l,R,\Delta}^c} \d \P_\infty - \int \tilde{f}_{\Delta,l} \1_{T_{k,l,R,\Delta}^c} \d \P_\infty \right| \leq \norm{f}_\infty \int \int \frac{\e^{-\tilde H_\Delta(\bx_\Delta' \cup \bx_{\Delta^c})}}{Z_\Delta(\bx_{\Delta^c})} \1_{I_{l,\Delta}} \mathbb{Q}^{\otimes \Delta} \d \P_\infty.
\end{equation*}
By dominated convergence theorem, we have for $l$ large enough
\begin{equation*}
    \int \int \frac{\e^{-\tilde H_\Delta(\bx_\Delta' \cup \bx_{\Delta^c})}}{Z_\Delta(\bx_{\Delta^c})} \1_{I_{l,\Delta}} \mathbb{Q}^{\otimes \Delta} \d \P_\infty \leq \epsilon.
\end{equation*}
Furthermore, for $k$  large enough we have
\begin{equation*}
    \left| \int \tilde{f}_{\Delta,l} \1_{T_{k,l,R,\Delta}^c} \d \P_\infty - \int \tilde{f}_{\Delta,l} \d \P_\infty \right|  \leq \norm{f}_\infty \P_\infty(T_{k,l,R,\Delta}) \leq \epsilon \norm{f}_\infty,
\end{equation*}
and this finishes the proof \eqref{ineq: dlr step6}.

\paragraph{Step 7}
Finally by dominated convergence theorem, for $l$ large enough we have
\begin{equation}\label{ineq: dlr step7}
    \left| \int \tilde{f}_{\Delta,l} \d \P_\infty - \int \tilde{f}_\Delta \d \P_\infty \right| \leq \epsilon \norm{f}_\infty.
\end{equation}
where
\begin{equation*}
    \tilde{f}_\Delta(\bx) = \frac{1}{Z_\Delta(\bx_{\Delta^c})}\int f(\bx_\Delta' \cup \bx_{\Delta^c}) \e^{-\tilde{H}_\Delta(\bx_\Delta' \cup \bx_{\Delta^c})} \mathbb{Q}^{\otimes \Delta}.
\end{equation*}

\paragraph{Conclusion} By combining \eqref{ineq: dlr step1}-\eqref{ineq: dlr step7}, we have
%, \eqref{ineq: dlr step2}, \eqref{ineq: dlr step3}, \eqref{ineq: dlr step4}, \eqref{ineq: dlr step5} and \eqref{ineq: dlr step6}, \eqref{ineq: dlr step7}, we have 
\begin{equation*}
    \left| \int f \d \P_\infty - \int \tilde{f}_\Delta \d \P_\infty \right| \leq 13 \epsilon \norm{f}_\infty.
\end{equation*}
This inequality, being true for all $\epsilon>0$, shows that $P_\infty$ satisfies the DLR equations for bounded local functions. Then, these DLR equations can be extended to bounded measurable functions by a monotone class argument, which ends this first part.\\

\noindent{\sc $\hookrightarrow$ \underline{Part B}. $\P_\infty$ satisfies DLR equations  under the assumptions~\ref{H:stability}-\ref{H:invariant} and \ref{H:summable}-\ref{Q:bounded}.} We consider the same notation and 7 steps as for the first part and detail only those that differ.

\paragraph{Step 2.bis}
For $n$ large enough, we have
\begin{equation}\label{ineq: dlr step2bis}
    \left| \int f \d \overline{\P}_n - \int \tilde{f}_\Delta \d \overline{\P}_n \right| \leq 2 \epsilon \norm{f}_\infty.
\end{equation}
This step follows from the same arguments as in Step 2 where we replace $f \1_{E_{k,l,R,\Delta}^c}$ with $f$.

\paragraph{Step 3.bis}
For $k>0$ large enough, we have
\begin{equation}\label{ineq: dlr step3bis}
    \sup_{n \in \N} \left| \int \tilde{f}_\Delta (\bx_{\Delta^c}) \d \overline{\P}_n -  \int \tilde{f}_\Delta (\bx_{\Delta_k \setminus \Delta}) \d \overline{\P}_n\right| \leq 3 \epsilon \norm{f}_\infty.
\end{equation}
By definition, the difference is equal to 
\begin{align*}
\int \tilde{f}_\Delta (\bx_{\Delta^c}) &- \tilde{f}_\Delta (\bx_{\Delta_k \setminus \Delta}) \d \overline{\P}_n = 
\int \int f(\bx) \frac{\e^{-\tilde{H}_\Delta(\bx_{\Delta_k \setminus \Delta} \cup \bx_\Delta')}}{Z_\Delta(\bx_{\Delta_k\setminus \Delta})} \times \\
& \hspace*{1cm}\left( \frac{\e^{-\sum_{i\in \Delta} \sum_{j \in \Delta_k^c} \phi(|i+x_i'-j-x_j|)} Z_\Delta(\bx_{\Delta_k \setminus \Delta}) - Z_\Delta(\bx_\Delta^c)}{Z_\Delta(\bx_\Delta^c)} \right) \d \mathbb{Q}^{\otimes \Delta} \d \overline{\P}_n.
\end{align*}
Since $\Qcal$ is bounded by~\ref{Q:bounded} and by monotonicity of $\phi$ at infinity, we have that for $k$ large enough, any $x_i',x_j \in \Qcal$
\begin{equation*}
    \sum_{i \in \Delta} \sum_{j \in \Delta_k^c} \left| \phi(|i+x_i'-j-x_j|) \right| \leq \sum_{i \in \Delta} \sum_{j \in \Delta_k^c} \left| \phi(|i-j| -D ) \right|,
\end{equation*}
where $D= \diam(\Qcal)$. According to assumption \ref{H:summable}, we know that $\phi$ is summable and thus for any $\eta>0$, there exists $k$ large enough such that 
\begin{equation*}
    \sum_{i \in \Delta} \sum_{j \in \Delta_k^c} \left| \phi(|i+x_i'-j-x_j|) \right| \leq \eta.
\end{equation*}
Therefore for $k$ large enough we have 
\begin{align*}
\Big| \int& \tilde{f}_\Delta (\bx_{\Delta^c}) \d \overline{\P}_n - \tilde{f}_\Delta (\bx_{\Delta_k \setminus \Delta}) \d \overline{\P}_n \Big| \\
\leq & \norm{f}_\infty \int \int \frac{\e^{-\tilde{H}_\Delta(\bx_{\Delta_k \setminus \Delta} \cup \bx_\Delta')}}{Z_\Delta(\bx_{\Delta_k\setminus \Delta})} \times \\
& \quad \left( \left| \frac{(\e^\eta -1) Z_\Delta(\bx_{\Delta_k \setminus \Delta})}{Z_\Delta(\bx_{\Delta^c})} \right| + \left| \frac{Z_\Delta(\bx_{\Delta_k \setminus \Delta}) - Z_\Delta(\bx_{\Delta^c}) }{Z_\Delta(\bx_{\Delta^c})}\right| \right) \d \mathbb{Q}^{\otimes \Delta} \d \overline{\P}_n \\
\leq & \norm{f}_\infty \int \left( \left| \frac{(\e^\eta -1) Z_\Delta(\bx_{\Delta_k \setminus \Delta})}{Z_\Delta(\bx_{\Delta^c})} \right| + \left| \frac{Z_\Delta(\bx_{\Delta_k \setminus \Delta}) - Z_\Delta(\bx_{\Delta^c}) }{Z_\Delta(\bx_{\Delta^c})}\right| \right) \d \mathbb{Q}^{\otimes \Delta} \d \overline{\P}_n.
\end{align*}
Furthermore, for $k$ large enough we have, 
\begin{align*}
    Z_\Delta(\bx_{\Delta_k \setminus \Delta}) =& \int \e^{-\tilde{H}_\Delta(\bx_{\Delta^c} \cup \bx_\Delta') + \sum_{i\in \Delta} \sum_{j \in \Delta_k^c} \phi(|i+x_i'-j-x_j|) } \d \mathbb{Q}^{\otimes \Delta} \\
    \leq & \e^\eta Z_\Delta (\bx_{\Delta^c}),
\end{align*}
and thus
\begin{align*}
    \left| \int \tilde{f}_\Delta (\bx_{\Delta^c}) \d \overline{\P}_n - \tilde{f}_\Delta (\bx_{\Delta_k \setminus \Delta}) \d \overline{\P}_n \right| \leq \norm{f}_\infty (\e^\eta - 1)(1 + \e^\eta). 
\end{align*}
Hence by choosing $\eta>0$ small enough and $k$ large enough we obtain \eqref{ineq: dlr step3bis}.

\paragraph{Step 4.bis}
For any $k$ fixed, by convergence in the local topology we have for $n$ large enough 
\begin{equation}\label{ineq: dlr step4bis}
    \left| \int \tilde{f}_\Delta (\bx_{\Delta_k \setminus \Delta}) \d \overline{\P}_n - \int \tilde{f}_\Delta (\bx_{\Delta_k \setminus \Delta}) \d \overline{\P}_\infty \right|  \leq  \epsilon \norm{f}_\infty.
\end{equation}

\paragraph{Step 5.bis}
By following the same argument as in Step 3.bis, we have for $k$ large enough
\begin{align}\label{ineq: dlr step5bis}
    \left| \int \tilde{f}_\Delta (\bx_{\Delta^c}) \d \overline{\P}_\infty -  \int \tilde{f}_\Delta (\bx_{\Delta_k \setminus \Delta}) \d \overline{\P}_\infty \right| \leq 3 \epsilon \norm{f}_\infty.
\end{align}

\paragraph{Conclusion}
By combining \eqref{ineq: dlr step1}, \eqref{ineq: dlr step2bis}, \eqref{ineq: dlr step3bis}, \eqref{ineq: dlr step4bis} and \eqref{ineq: dlr step5bis}, we have
\begin{equation*}
    \left| \int f \d \P_\infty - \int \tilde{f}_\Delta \d \P_\infty \right| \leq 10 \epsilon \norm{f}_\infty
\end{equation*}
and we conclude in the same way as for the first part.\\

\noindent{\sc $\hookrightarrow$ \underline{Part C}. Stationarity of $\bG$ and $\bG\mid U=u$}. Parts A-B prove that $\Gcal_\Lcal(H, \mathbb Q) \neq \emptyset$ and as a consequence we have that $P_\Lcal(H,\mathbb Q) \neq \emptyset$. All that is left is to prove that the Gibbs perturbed lattice $\bG$ is a stationary point process. We start with a first observation, by definition for $\P \in \Gcal_\Lcal(H, \mathbb{Q})$ and $\bX \sim \P$ the law of $\bX$ is invariant under translation and therefore for all $i\in \Lcal$, $X_i \overset{d}{=} X_0$. We consider $W \subset \R^d$ bounded, since $(X_i)_{i\in \Lcal}$ are identically distributed we have that
\begin{align*}
\E[\# (\bG \cap W)] &= \sum_{i\in \Lcal} \mathrm P(i+X_i+U \in W) \\
&\leq \sum_{i\in \Lcal} \mathrm P(X_i \in ( W \oplus \Dcal)-i ) \\
&\leq \sum_{i\in \Lcal} \mathrm P(X_0 \in ( W \oplus \Dcal)-i ) \\
   &\leq |W \oplus \Dcal|. 
\end{align*} 
Since $\E(\# \bG \cap W)$  is finite then almost surely $\# \bG \cap W$ is finite and thus $\bG$ is almost surely locally finite. Now we consider $ t \in  \R^d$, then there is a unique $k \in \Lcal$ and $V \in \Dcal$ such that $t+U= k+V$ and we know that $V \sim \mathcal{U}_\Dcal$. Therefore, for any bounded measurable function $f$ on $\Ccal$ we have
\begin{align*}
\E\left[f(\bG + t)\right] & = \E\left[ f(\bG(\bX,t+U)) \right] \\
&=\E\left[ f(\bG(\bX,k+V)) \right] \\
& = \E\left[ f(\bG(\bX,V)) \right] = \E\left[ f(\bG )\right].
\end{align*}
Therefore, for any $t \in\R^d$ we have $\bG = \bG + t$. Since $\bG$ is almost surely locally finite and its distribution is invariant under translations of $\R^d$, it follows that $\bG$ is a stationary point process. Finally, for any $u \in \Dcal$ and $k \in \Lcal$, we also have 
\begin{equation*}
\E_u\left[f(\bG + k)\right] = \E_u\left[f(\bG (\bX, k + u))\right] = \E_u\left[f(\bG (\bX, u))\right] = \E_u\left[f(\bG )\right].  
\end{equation*}
Therefore, for any $u \in \Dcal$ the law of $\bG \mid U=u$ is invariant under translation by any element of $\Lcal$.
\end{proof}

\section{Proof of Proposition~\ref{prop:HU}}
\label{app:hyperuniformity}

\begin{proof}   
(i) Let  $\mathbb P\in \Lcal(H, \mathbb{Q})$. According to Lemma \ref{lem: specific entropy bound} we have that 
\begin{equation}
\E[|X_0|^d] \leq c \left( I_s(\mathbb P) + \log \E[\e^{c|X|^d}] \right)
\end{equation}
for $X \sim \mathbb Q$. It is not clear whether for any $\mathbb P \in P\Lcal(H, \mathbb{Q})$,  $I_s(\mathbb P)<\infty$. However we know by \eqref{ineq: bounded specific entropy} that we have at least one measure $\mathbb P_\infty$ with a finite specific entropy. Hence, $\bG \sim \mathbb P_\infty$ is hyperuniform.

\noindent(ii) By DLR equations, for any $\mathbb P \in P\Lcal(H, \mathbb{Q})$, we have
\begin{equation*}
\E[|X_0|^d] = \E\left[ \int |x|^d \frac{\e^{-h(x + U, \bG_{0^c})}}{Z_0(\bG_{0^c})} \mathbb{Q}(\d x) \right]
\end{equation*}
where $\bG_{0^c} = \{i + X_i + U , i \in \Lcal\setminus\{0\} \} $. Under the assumption~\ref{H:bounded}, there exists $M \geq 0$  such that $|h| \leq M$ and thus we have 
\begin{gather*}
\E[|X_0|^d] \leq \e^{2M} \E[|X|^d] < \infty
\end{gather*}
for $X \sim \mathbb Q$. 
\end{proof}

\section{Proofs of results related to DLR type equations} \label{app:DLR}

This Appendix focuses on the proofs of Propositions~\ref{prop:sumDLR}-\ref{prop:variational}.

\subsection{Proofs of Proposition~\ref{prop:sumDLR}}

\begin{proof}
Proposition~\ref{prop:sumDLR}(i) follows directly from summing~\eqref{eq:onesiteDLR} over all sites. Statement~(ii) is then a consequence obtained by using the stationarity of $\bG\mid U=u$ and the invariance by translation of $f$.
\end{proof}

\subsection{Proof of Proposition~\ref{prop:sum2sitesDLR}}

\begin{proof}
This result is obtained by applying Proposition~\ref{prop:sumDLR} twice.
\end{proof}

\subsection{Proof of Proposition~\ref{prop:variational}} 

\begin{proof}
Let $u \in \Dcal$ and $ i \in \Lcal_u$. From~\eqref{eq:onesiteDLR} we have
\begin{align*}
\E_u\left[ \grad \check f(i+X_i, \bG_{i^c}) \right] = \E_u \left[ \frac{1}{Z_i(\bG_{i^c})} \int_\Qcal \grad \check f(i+x, \bG_{i^c}) \Lambda(i,x, \bG_{i^c}) \d x \right].
\end{align*}
Using an integration by parts on the right-hand side term, we have
\begin{gather*}
\E_u\left[ \grad \check f(i+X_i, \bG_{i^c}) \right] = \E_u \left[ \frac{1}{Z_i(\bG_{i^c})} \oint_{\partial \Qcal} \check f(i+x, \bG_{i^c}) \e^{-h(i+x, \bG_{i^c}) -q(x)} \d x\right] \\
+ \E_u \left[ \frac{1}{Z_i(\bG_{i^c})} \int_\Qcal \check f(i+x, \bG_{i^c}) (\grad h(i+x, \bG_{i^c}) + \grad q(x)) \e^{- h(i+x,\bG_{i^c}) - q(x)} \d x \right]. 
\end{gather*}
Since $\check f$ belongs to $\Fcal_\Rcal(H, \mathbb{Q}, \Lcal_u)$, the first term of the previous equation equals 0. Then, by applying the DLR equations to the second term, we obtain
\begin{equation*}
\E_u\left[ \grad \check f(i+X_i, \bG_{i^c}) \right] = \E_u \left[ \check f(i+X_i, \bG_{i^c}) (\grad h(i+X_i, \bG_{i^c}) + \grad q(X_i) ) \right]
\end{equation*}
The result is proved by summing over all sites of $\Lcal_u$ and by applying Fubini's theorem, justified by~\eqref{condition: fubini_eq_variationnel}.
\end{proof}

\section{Proof of Theorem~\ref{thm:ergodic}} \label{app:ergodic}

The proof of this result is mainly based on the following ergodic theorem obtained by~\citep{nguyen1977integral}, which we recall here for the sake of self-consistency.

\begin{theorem}[Ergodic theorem \citep{nguyen1977integral}] \label{thm:ergodicNZ}
Let $\Lcal$ be a full rank lattice and $\P$ be an ergodic measure on $\Ccal$. Let $F_W$ be a family of measurable functions, indexed by bounded sets $W \in \R^d$, which are additive, i.e. $F_{W \cup W'} = F_W + F_{W'} - F_{W\cap W'}$, shift invariant on $\Lcal$, i.e. $F_W(\bg) = F_{W+i}(\bg+i)$ for any $i \in \Lcal$, and integrable, i.e. $\E[|F_{[0,1]^d}|] < \infty$. Then, we have that for $\P$-a.e. $\bg$,
    \begin{equation}
        \lim\limits_{n \to \infty} |W_n|^{-1} F_{W_n}(\bg) = \E[|F_{[0,1]^d}]
    \end{equation}
    where $W_n$ is a  sequence of regular bounded domains of $\R^d$ such that $W_n\to \R^d$ as $n\to \infty$. 
\end{theorem}

%\begin{theorem}[Ergodic theorem on a lattice] \label{thm:ergodiclattice} 
%    Let $P$ an ergodic measure on $(\R^d)^\Lcal$ then for any measurable function $f$ such that $E_P[|f|]<\infty$ we have that for $P$-a.e. $\bg$
%    \begin{equation}
%        \lim\limits_{n \to \infty} |\Lambda_n|^{-1} \sum_{i \in \Lambda_n} f(\bg+i) = \E_P[f]
%    \end{equation}
%    where $\Lambda_n = [-n,n]^d \cap \Lcal$.
%\end{theorem}

\begin{proof}
(i) Let $u \in \Dcal$. Due to the decomposition of stationary measures as a mixture of ergodic measures, see \citep{preston1976random}, and by Theorem~\ref{thm:existence}, it suffices to prove Theorem~\ref{thm:ergodic} by assuming that $\P_u$ is ergodic. 
We introduce the following random variables.
\begin{align*}
%A_{W_n} &= \sum_{i\in \Lcal_U} b_n(i,X_i)f_l(i,X_i,\bG_{i^c};\theta) \\
A_{1,W_n} &= \sum_{i\in \Lcal_U}  \1(i+X_i\in W_n\ominus R)f(i,X_i,\bG_{i^c};\theta) \\
A_{2,W_n} &= \sum_{i\in \Lcal_U}  \1(i+X_i\in W_n\ominus R, X_i \in \underline{M}_n, i\in (W_n\ominus(m_n+R))^c)f(i,X_i,\bG_{i^c};\theta) \\
A_{3,W_n} &= \sum_{i\in \Lcal_U}  \1(i+X_i\in W_n\ominus R, X_i \in \overline{M}_n, i\in (W_n\ominus(m_n+R))^c)f(i,X_i,\bG_{i^c};\theta) 
\end{align*}
and
\begin{align*}
%B_{W_n} &= \sum_{i\in \Lcal_U} \int_{\Qcal} %b_n(i,x)f_l(i,x,\bG_{i^c};\theta)\Lambda_n(i,x,\bG_{i^c};\theta) \d x \\
%\widetilde{B}_{W_n} &= \sum_{i\in \Lcal_U} \int_{\Qcal} %b_n(i,x)f_l(i,x,\bG_{i^c};\theta)\Lambda(i,x,\bG_{i^c};\theta) \d x \\
\widetilde B_{1,W_n} &= \sum_{i\in \Lcal_U} \int_{\Qcal} \1(i\in W_n \ominus (m_n+R)f(i,x,\bG_{i^c};\theta)\Lambda(i,x,\bG_{i^c};\theta) \d x \\
\widetilde B_{2,W_n} &= \sum_{i\in \Lcal_U} \int_{\Qcal} \1(i\in W_n \ominus(m_n+R), x \in \underline{M_n}, i+x \in (W_n\ominus R)^c) \times \\
&\qquad \qquad f(i,x,\bG_{i^c};\theta)\Lambda(i,x,\bG_{i^c};\theta) \d x \\
\widetilde B_{3,W_n} &= \sum_{i\in \Lcal_U} \int_{\Qcal} \1(i\in W_n \ominus(m_n+R), x \in \overline{M_n}, i+x \in (W_n\ominus R)^c) \times \\
&\qquad \qquad f(i,x,\bG_{i^c};\theta)\Lambda(i,x,\bG_{i^c};\theta) \d x 
\end{align*}
where for the sequence $(m_n)_n$ defined in~\ref{M:model}, we let $\overline M_n = \{x \in \Qcal: |x|\ge m_n\}$ and $\underline{M}_n = \Qcal \setminus \overline M_n$.
We can observe, by shortening a little bit the notation, that
\begin{align*}
A_{W_n}:= A_{W_n}(f,\bG;\theta) &= A_{1,W_n} - A_{2,W_n} -A_{3,W_n}\\
\widetilde{B}_{W_n}:= \widetilde{B}_{W_n}(f,\bG;\theta) &=\widetilde{B}_{1,W_n} -\widetilde{B}_{2,W_n} -\widetilde{B}_{3,W_n} \\
\mathrm{DLR}_n := \mathrm{DLR}_n(f, \bG;\theta) &= A_{W_n} - B_{W_n} \\
\widetilde{\mathrm{DLR}}_n := \widetilde{\mathrm{DLR}}_n(f, \bG;\theta) &= A_{W_n} - \widetilde{B}_{W_n} \\
\mathrm{DLR}_n - \widetilde{\mathrm{DLR}}_n &= B_{W_n} - \widetilde B_{W_n}.
\end{align*}
We point out the following facts: for any $i\in \Lcal_U$ and $x\in \Qcal$,
\begin{align}
\1(i+X_i\in W_n\ominus R, X_i \in \underline{M}_n, i\in (W_n\ominus(m_n+R))^c)&=0\label{eq:trick1a}\\
\1(i\in W_n \ominus(m_n+R), x \in \underline{M_n}, i+x \in (W_n\ominus R)^c)&=0\label{eq:trick1b}\\
\1(i\in W_n \ominus (m_n+R), x \in \overline{M_n}, i+x \in (W_n\ominus R)^c)& \le \1(i\in W_n) \label{eq:trick2}
\end{align}
and that as $n\to \infty$
\begin{equation}\label{eq:trick3}
|W_n| \sim |W_n\ominus R|\sim |W_n \ominus(m_n+R)|
\text{ and }
|W_n \oplus m_n \setminus W_n| \sim 2d m_n |W_n|^{1-1/d}.
\end{equation}
In particular, we can observe that $A_{2,W_n}=\widetilde B_{2,W_n}=0$.

The proof of Theorem~\ref{thm:ergodic}(i) reduces to showing that, as $n\to \infty$ and given $U=u$, almost surely  $|W_n|^{-1}A_{1,W_n}$ and $|W_n|^{-1}\widetilde B_{1,W_n}$ tend to~\eqref{eq:limA}-\eqref{eq:limB} respectively, and that $|W_n|^{-1}C_{W_n}\to 0$ for $C_{W_n}=A_{3,W_n}, \widetilde B_{3,W_n}$ and $B_{W_n}-\widetilde B_{W_n}$.

\paragraph{Proof that $|W_n|^{-1}A_{1,W_n}$ and $|W_n|^{-1}\widetilde B_{1,W_n}$ tend to~\eqref{eq:limA}-\eqref{eq:limB} respectively}

We first focus on the term $A_{1,W_n}$. Let $\tilde \bG$ be the marked point process defined by $\tilde \bG=\{(i+X_i,X_i), i\in \Lcal_U\}$. By assumption on $\bG$ and $f$, one can rewrite $A_{1,W_n}$ as 
\[
A_{1,W_n} = \sum_{z=(y,x) \in \tilde \bG \cap ((W_n\ominus R)) \times \Qcal)}  f_1(x)f_2(y,\tilde \bG \setminus z) = 
\sum_{z\in \tilde \bG \cap (W_n\ominus R) \times \Qcal)}  \tilde f (z,\tilde \bG)
\]
where we observe that $\tilde f$ is shift invariant with respect to the first coordinate of $z$. By assumption, $A_1$ is furthermore additive and such that $\E_u[|A_{1,[0,1]^d}|]<\infty$. We are in a position to apply Theorem~\ref{thm:ergodicNZ}, which leads to  $\P_u$-a.s. as $n\to \infty$, $|W_n|^{-1}A_{1,W_n} \to \Eu[A_{1,[0,1]^d}]$. Finally, using Proposition~\ref{prop:sumDLR}(ii), observe that 
\begin{align*}
\E_u[A_{1,[0,1]^d}] %&= %\E_u \left[ \sum_{i\in \Lcal_U} f_{[0,1]^d} (i,X_i,\bG_{i^c}) \right]\\ 
%&=\E_u \left[ \sum_{i\in \Lcal_U} \int_{\Qcal} f_{[0,1]^d} (i,x,\bG_{i^c}; \theta) \Lambda (i,x,\bG_{i^c}; \theta^\star) \d x \right]\\
& = \sum_{i \in \Lcal_u} \int_{\Qcal} \1_{[0,1]^d}(i+x) \E_u \left[ 
 f(i,x,\bG_{i^c}; \theta) \Lambda (i,x,\bG_{i^c}; \theta^\star)\right] \d x \\
& = \sum_{i \in \Lcal_u} \int_{\Qcal} \1_{[0,1]^d}(i+x) \E_u \left[ 
 f(0,x,\bG_{0^c}; \theta) \Lambda (0,x,\bG_{0^c}; \theta^\star)\right] \d x \\
 & =  \int_{\Qcal}  \E_u \left[
 f(0,x,\bG_{0^c};\theta) \Lambda (0,x,\bG_{0^c};\theta^\star) \right] \d x .
\end{align*}
For the term, $\widetilde{B}_{1,W_n}$, assumptions of Theorem~\ref{thm:ergodic} allow us to apply Theorem~\ref{thm:ergodicNZ} to $\widetilde B_{1,W_n}$. Hence, as $n\to \infty$ and $\P_u$-a.s.
\begin{equation*}
|W_n|^{-1} \widetilde B_{1,W_n} {\longrightarrow} \E_u \left[ \int_{\Qcal}  
 f(0,x,\bG_{0^c};\theta) \Lambda (0,x,\bG_{0^c};\theta) \d x \right]. 
\end{equation*}

%\paragraph{Proof that %$|W_n|^{-1}A_{2,W_n}\stackrel{a.s.}{\to} 0$} Using \ref{Af}, \eqref{eq:trick1} and \eqref{eq:trick3}, there exists $c>0$ such that \begin{align*}
%|W_n|^{-d-1}\Eu|A_{2,W_n}|^{d+1} &\leq c \; \frac{(\#\Lcal_u \cap (W_n\oplus m_n \setminus W_n))^{d+1}}{|W_n|^{d+1}} \\
%&= O \left( \frac{|W_n\oplus m_n \setminus W_n|^{d+1}}{|W_n|^{d+1}}\right)\\
%&= O(m_n^{d+1} |W_n|^{-1-1/d} )\\
%&= O(   |W_n|^{-1-1/d} \; \log(|W_n|)^{d+1}).
%\end{align*}
%Now, we use~\ref{M:model}, Markov's inequality and %Borel-Cantelli's lemma to conclude this step.

\paragraph{Proof that $|W_n|^{-1}A_{3,W_n}\stackrel{a.s.}{\to} 0$} Using the stationarity of $\bG \mid U=u$, \ref{Af} and~\eqref{eq:trick2}, there exists $c>0$ such that
\begin{align*}
|W_n|^{-1} \Eu |A_{3,W_n}| &\le c |W_n|^{-1} \# (\Lcal_u \cap W_n) \int_{\overline M_n} \lambda_1(x;\theta^\star) \d x \\
&\le c \; \P(X \ge m_n).
\end{align*}
for $X\sim \mathbb Q$. Using \ref{M:model} and Markov's inequality, it is clear that $\P(X \ge m_n)=O(e^{-m_n})= O(|W_n|^{-\beta})$. Finally, Markov's inequality,~\ref{M:model} and Borel-Cantelli's lemma conclude this step.

\paragraph{Proof that $|W_n|^{-1}\widetilde B_{3,W_n}\stackrel{a.s.}{\to} 0$} 

This result is proved along similar lines as the previous result, except we now combine~\ref{Bf}, \ref{M:model} and \eqref{eq:trick1b}-\eqref{eq:trick3}.

\paragraph{Proof that $|W_n|^{-1}(B_{W_n}-\widetilde B_{W_n})\stackrel{a.s.}{\to} 0$}

We first observe that

\begin{align}
|B_{W_n}-\widetilde B_{W_n}| &\le \int_{\Qcal}\int_{\Qcal}  Z_n(i,x,y,\bG_{i^c};\theta) (1-b_n(i,y)) \d x \d y \label{eq:BBtilde}
\end{align}
with
\begin{equation*}
Z_n(i,x,y,\bG_{i^c};\theta) = b_n(i,x) |f(i,x,\bG_{i^c}; \theta)| \Lambda(i,x,\bG_{i^c};\theta)\Lambda_n(i,y,\bG_{i^c};\theta)
\end{equation*}
Now, we decompose the right-hand side of~\eqref{eq:BBtilde} as $\Delta_{\underline M_n}+\Delta_{\overline M_n}$, where
\begin{align*}
\Delta_{\underline M_n}&=\int_{\Qcal}\int_{\Qcal}  Z_n(i,x,y,\bG_{i^c};\theta) \times\\
&\qquad \qquad 
\1(i\in W_n\ominus(m_n+R), y\in \underline{M_n}, i+y \in (W_n\ominus R)^c) \d x \d y.\\
\Delta_{\overline M_n}&= \int_{\Qcal}\int_{\Qcal}  Z_n(i,x,y,\bG_{i^c};\theta) \times \\
&\qquad\qquad\1(i\in W_n\ominus(m_n+R), y\in \overline{M_n}, i+y \in (W_n\ominus R)^c) \d x \d y.
\end{align*}
Using~\eqref{eq:trick1b}, $\Delta_{\underline M_n}=0$.
We leave the reader to check that using \ref{Bf} and the same arguments as before, we have that
\begin{align*}
%    |W_n|^{-d-1} \Eu|\Delta_{\underline M_n}|^{d+1} &= O( |W_n|^{-d-1} \; \log(|W_n|^{d+1}))\\
    |W_n|^{-1} \Eu|\Delta_{\overline M_n}| &= O( |W_n|^{-\beta})
\end{align*}
whereby the result is deduced using Markov's inequality and Borel-Cantelli's lemma applied to $|W_n|^{-1}\Delta_{\overline M_n}$. \\

%$|W_n|^{-1}\Delta_{\underline M_n}$ and 

(ii) Finally, to prove~\eqref{DLRDLRtildeprobability}, we observe that the previous developments also prove that
\begin{align*}
%  \Delta_{\underline M_n} &= O_{\P_u}(\E_u|\Delta_{\underline M_n}|) =     O_{\P_u}( |W_n|^{1-1/d} \log |W_n| ) \\
 \Delta_{\underline M_n} &= O_{\P_u}(\E_u|\Delta_{\overline M_n}|) =   O_{\P_u}( |W_n|^{1-\beta} ).
\end{align*}
Therefore,
%$B_{W_n}-\widetilde B_{W_n}=O_{\P_u}( |W_n|^{1-1/d} \log |W_n| )$ 
$B_{W_n}-\widetilde B_{W_n}=O_{\P_u}( |W_n|^{1-\beta}  )$ 
which yields the result.
\end{proof}

\section{Proof of Theorem~\ref{thm:convTF}}  \label{app:convTF}

\subsection{Auxiliary result}

\begin{lemma}\label{lem:DLRk}
Assume~\ref{M:model}-\ref{M:ergodic}. For $\ell=0,1,2$, $l=1,\dots,L$ and $\theta\in \Theta$, we have $\P_u$-a.s. as $n\to \infty$
\begin{equation}\label{eq:DLRk}
    |W_n|^{-1} \mathrm{DLR}^{(\ell)}_{n}(f_l,\bG; \theta) \stackrel{a.s.}{\longrightarrow} \Dcal^{(\ell)}(f_l;\theta)
\end{equation}
where
\begin{align*}
\Dcal^{(0)}(f_l;\theta) = \Dcal(f_l;\theta) &= \Eu 
\left[ 
\int_{\Qcal} f_l(0,x,\bG_{0^c}; \theta) \left( 
\Lambda(0,x,\bG_{0^c};\theta^\star)-\Lambda(0,x,\bG_{0^c};\theta)
\right) \d x
\right]\\
\Dcal^{(1)}(f_l; \theta) &= \Eu 
\left[ 
\int_{\Qcal} f_l^{(1)}(0,x,\bG_{0^c}; \theta) \left( 
\Lambda(0,x,\bG_{0^c};\theta^\star)-\Lambda(0,x,\bG_{0^c};\theta)
\right) \d x
\right]\\
&\quad - \Eu 
\left[ 
\int_{\Qcal} f_l(0,x,\bG_{0^c}; \theta) \Lambda^{(1)}(0,x,\bG_{0^c};\theta) \d x
\right]\\
\Dcal^{(2)}(f_l; \theta) &= 
\Eu 
\left[ 
\int_{\Qcal} f_l^{(2)}(0,x,\bG_{0^c}; \theta) \left( 
\Lambda(0,x,\bG_{0^c};\theta^\star)-\Lambda(0,x,\bG_{0^c};\theta)
\right) \d x
\right]\\
& \quad - 
\Eu 
\left[ 
\int_{\Qcal} 
f_l^{(1)}(0,x,\bG_{0^c}; \theta) \Lambda^{(1)}(0,x,\bG_{0^c};\theta)^\top  \d x \right]\\
&\quad - 
\Eu 
\left[ 
\int_{\Qcal} 
\Lambda^{(1)}(0,x,\bG_{0^c}; \theta) f_l^{(1)}(0,x,\bG_{0^c};\theta)^\top  \d x \right]\\
&\quad - 
\Eu 
\left[ 
\int_{\Qcal} 
f_l(0,x,\bG_{0^c}; \theta) \Lambda^{(2)}(0,x,\bG_{0^c};\theta)^\top  \d x \right].
\end{align*}
In particular, we have $\Dcal(f_l;\theta^\star)=0$ and $\Dcal^{(1)}(f_l; \theta^\star) = - \mathcal I(f_l;\theta^\star)$, where $\mathcal I(f_l;\theta^\star)$ is given by~\eqref{eq:Il}.
\end{lemma}

\begin{proof}
The proof of this result is a direct application of Theorem~\ref{thm:ergodic}, valid under the assumptions~\ref{M:model}-\ref{M:ergodic}.
\end{proof}

\subsection{Proof of Theorem~\ref{thm:convTF}}

\begin{proof} (i) We have, by Proposition~\ref{prop:sumDLR}, that
\begin{align*}
\Var_u\left[ \widetilde{\mathrm{DLR}}_n(f_l,\bG;\theta^\star)\right] &= \Eu \left[ \widetilde{\mathrm{DLR}}_n(f_l;\theta^\star)^2\right]\\
&= \E \Big[
\Big\{
\sum_{i \in \Lcal_U} b_n(i,X_i)f_{l}(i,X_i, \bG_{i^c}) \\
& \qquad - 
\int  b_n(i,x)f_{l}(i,x, \bG_{i^c})  \Lambda(i,x,\bG_{i^c}) \d x
\Big\}^2
\Big]
\end{align*}
where to abuse and ease notation we let $f_{l}(i,x,\bg) =f_{l}(i,x,\bg;\theta^\star)$, $\Lambda(i,x,\bg) =\Lambda(i,x,\bg; \theta^\star)$ and $\int=\int_{\Qcal}$. Using Propositions~\ref{prop:sumDLR} and~\ref{prop:sum2sitesDLR}, we decompose the conditional variance $$\Var_u\left[ \widetilde{\mathrm{DLR}}_n(f_l,\bG;\theta^\star)\right]=T_1+T_2$$ where
\begin{align*}
T_1  &= \Eu \Big[
\sum_{i \in \Lcal_U} \int b_n(i,x)f_{l}(i,x,\bG_{i^c})^2\Lambda(i,x,\bG_{i^c})\d x \\
&\qquad \qquad \qquad -  
\Big\{
\int b_n(i,x) f_{l}(i,x,\bG_{i^c})\Lambda(i,x,\bG_{i^c})\d x
\Big\}^2
\Big]
\\
T_2 & = \Eu \left[
\sum_{i,j \in \Lcal_U}^{\neq} \int\!\!\!\!\int
b_n(i,x)b_n(j,y) f_{l}(i,x,\bG_{i^c}) f_{l}(j,y,\bG_{j^c})
\Lambda(i,x,\bG_{i^c}) \Lambda(j,y,\bG_{j^c})
\Delta \d x \d y
\right]
\end{align*}
and
\begin{equation*}
\Delta = \Delta(i,x,j,y,\Gamma_{i^c},\Gamma_{j^c},\Gamma_{(ij)^c})= \frac{\Lambda(i,x,\bG_{(ij)^c})\Lambda(j,y,\bG_{(ij)^c} \cup \{i+x\})}{\Lambda(i,x,\bG_{i^c}) \Lambda(j,y,\bG_{j^c})} -1.     
\end{equation*}
Regarding the term $T_1$, using the stationarity of $\bG \mid U=u$, Fubini's theorem and by Assumption~\ref{M:ergodic}, we have the simple upper-bound
\begin{align}
T_1 &\leq  \int \# (\Lcal_u \cap (W_n \ominus (m_n+R ))) \; \Eu \left[ 
|f_l(0,x,\bG_{0^c})|\Lambda(0,x,\bG_{0^c})\right] \d x = O(|W_n|). \label{eq:asympT1}
\end{align}
Let us consider the term $T_2$. Let $a,b \in \R^d$ and $\bg \in \Ccal$. By definition of $\lambda_2$, we have
\begin{equation*}
    \lambda_2(a\cup b ,\bg ) =\lambda_2(a ,\bg) \lambda_2(b,\bg \cup a) =\lambda_2(b ,\bg) \lambda_2(a,\bg \cup b) 
\end{equation*}
which can be rewritten as
\begin{equation}\label{eq:tricklambda2}
\frac{\lambda_2(a ,\bg)}{\lambda_2(a,\bg \cup b)} = \frac{\lambda_2(b, \bg)}{\lambda_2(b,\bg \cup a)}.
\end{equation}
Using this, we rewrite $\Delta+1$ as
\begin{align*}
\Delta+1 &= A(i,x,X_i,j,y,X_j,\bG_{i^c},\bG_{j^c}, \bG_{(ij)^c})  \times \\ 
& \qquad \frac{\int\int \lambda(i,\tilde x, \bG_{i^c})\lambda(j,\tilde y, \bG_{j^c} )\d \tilde x \d \tilde y}{\int\int \lambda(i,\tilde x, \bG_{i^c})\lambda(j,\tilde y, \bG_{j^c} )
A(i,\tilde x,X_i,j, \tilde y,X_j,\bG_{i^c},\bG_{j^c}, \bG_{(ij)^c})
\d \tilde x \d \tilde y } 
\end{align*}
where for any $i,j \in \Lcal_U$, $x,x^\prime,y,y^\prime \in \Qcal$, we have, using the definition of $\Lambda$ 
\begin{align*}
A(i,x,X_i,j,y,X_j,\bG_{i^c},\bG_{j^c}, \bG_{(ij)^c}) &= \frac{\lambda_2(i+x,\bG_{(ij)^c})}{\lambda_2(i+x,\bG_{i^c})}
\frac{\lambda_2(j+y,\bG_{(ij)^c}\cup \{i+x\})}{\lambda_2(j+y,\bG_{j^c})} 
\end{align*}
Using~\eqref{eq:tricklambda2}, this term can be rewritten as
\begin{align} 
A(i,x,X_i,j,y,X_j,\bG_{i^c},\bG_{j^c}, \bG_{(ij)^c}) &= \frac{\lambda_2(j+X_j,\bG_{(ij)^c}}{\lambda_2(j+X_j,\bG_{(ij)^c)} \cup \{i+x\})}\nonumber \\
&\qquad \times
\frac{\lambda_2(i+X_i,\bG_{(ij)^c} \cup \{i+x\})}{\lambda_2(i+X_i,\bG_{(ij)^c} \cup \{i+x\} \cup \{j+y\})} \nonumber\\
&\qquad \times
\frac{\lambda_2(j+y, \bG_{j^c}\cup\{i+x\})}{\lambda_2(j+y, \bG_{j^c})}. \label{eq:A}
\end{align}
Hence, we view the term $T_2$ as a general functional of $i,X_i,j,X_j$ and $\bG$, to which we can apply Proposition~\ref{prop:sum2sitesDLR} once more. We obtain that 
\begin{align*}
T_2 &= \Eu \sum_{i,j \in \Lcal_U}^{\neq} \int\int\int\int b_n(i,x)b_n(j,y) 
F_l(i,x,x^\prime,j,y, y^\prime, \bG_{i^c},\bG_{j^c},\bG_{(ij)^c}) \times\\
& \hspace*{2cm} \Delta(i,x,x^\prime,j,y, y^\prime, \bG_{i^c},\bG_{j^c},\bG_{(ij)^c}) 
\lambda_1(x)\lambda_1(x^\prime)\lambda_1(y)\lambda_1(y^\prime) \d x \d x^\prime \d y \d y^\prime
\end{align*}
where
\begin{align*}
F_l(i,x,x^\prime,j,y, y^\prime, \bG_{i^c},\bG_{j^c},\bG_{(ij)^c}) &= 
f_l(i,x,\bG_{i^c})\lambda_2(i+x,\bG_{i^c})  \\
&\quad \times f_l(j,y,\bG_{j^c})\lambda_2(j+y,\bG_{j^c}) \\
&\quad\times \lambda_2(i+x^\prime , \bG_{(ij)^c}) \lambda_2(j+y^\prime, \bG_{(ij)^c} \cup \{i+x^\prime\}).
\end{align*}
Similarly to the term $T_1$, we use Fubini's theorem and the stationarity of $\Gamma \mid U=u$ to derive
\begin{align*}
T_2 &= \# (\Lcal_u \cap (W_n\ominus (m_n+R))) \times \\
& \quad \Eu \sum_{j \in \Lcal_U \setminus \{0\} } \int\int\int\int \1(x \in W_n \ominus R) b_n(j,y) F_l(0,x,x^\prime,j, y, y^\prime, \bG_{0^c},\bG_{j^c},\bG_{(0j)^c}) \times\\
& \hspace*{2cm} \Delta(0,x,x^\prime,j, y, y^\prime, \bG_{0^c},\bG_{j^c},\bG_{(0j)^c}) 
\lambda_1(x)\lambda_1(x^\prime)\lambda_1(y)\lambda_1(y^\prime) \d x \d x^\prime \d y \d y^\prime \\
&= \# (\Lcal_u \cap W_n\ominus(m_n + R)) \times (T_1^\prime +T_2^\prime)
\end{align*}
where
\begin{align*}
T_1^\prime & = \Eu \sum_{j \in \Lcal_U} \int\int\int\int \1(x \in W_n \ominus R) b_n(j,y)  \\
& \hspace*{2cm}F_l(0,x,x^\prime,y, y^\prime, \bG_{0^c},\bG_{j^c},\bG_{(0j)^c}) \times R_1(0,x,x^\prime,j,y,y^\prime , \bG_{0^c},\bG_{j^c},\bG_{(0j)^c}) \times \\ 
& \hspace*{2cm} \left\{ A(0,x,x^\prime,j,y,y^\prime , \bG_{0^c},\bG_{j^c},\bG_{(0j)^c})  -1\right\}
\d x \d x^\prime \d y \d y^\prime\\
T_2^\prime &=\Eu \sum_{j \in \Lcal_U} \int\int\int\int \int\int \1(x \in W_n \ominus R) b_n(j,y)\\
& \hspace*{2cm}F_l(0,x,x^\prime,y, y^\prime, \bG_{0^c},\bG_{j^c},\bG_{(0j)^c})
\times R_2(0,\tilde x,x^\prime,j,\tilde y,y^\prime, \bG_{0^c},\bG_{j^c},\bG_{(0j)^c}) \times \\ 
& \hspace*{2cm} \left\{ A(0,\tilde x,x^\prime,j,\tilde y,y^\prime , \bG_{0^c},\bG_{j^c},\bG_{(0j)^c})  -1\right\}
 \lambda_1(\tilde x)\lambda_1(\tilde y) \d x \d x^\prime \d \tilde x \d y \d y^\prime \d \tilde y \\
\end{align*}
where 
\begin{align*}
R_1 &= \frac{\int\int \lambda(0,\tilde x, \bG_{0^c})\lambda(j,\tilde y, \bG_{j^c} )\d \tilde x \d \tilde y}{\int\int \lambda(0,\tilde x, \bG_{0^c})\lambda(j,\tilde y, \bG_{j^c} )
A(0,\tilde x,x^\prime,j, \tilde y,y^\prime,\bG_{0^c},\bG_{j^c}, \bG_{(0j)^c})
\d \tilde x \d \tilde y } 
\end{align*}
and 
\begin{align*}
R_2 &=
\frac{\lambda_2(\tilde x, \bG_{0^c}) \lambda_2(j+\tilde y, \bG_{j^c})}
{\int\int \lambda(0,\tilde x, \bG_{0^c})\lambda(j,\tilde y, \bG_{j^c} )
A(0,\tilde x,x^\prime,j, \tilde y,y^\prime,\bG_{0^c},\bG_{j^c}, \bG_{(0j)^c})
\d \tilde x \d \tilde y }. 
\end{align*}
We are now in a position to finally use the finite range property~\ref{H:range}. Indeed, for any $x,x^\prime,y,y^\prime \in \Qcal$, any $j\in \Lcal_U$, we can observe that $A(0,x,x^\prime,j,y,y^\prime,   \bG_{0^c},\bG_{j^c}, \bG_{(0j)^c}) = 1$  if $j \notin \mathcal B_R(x,x^\prime,y,y^\prime)$ defined by
\begin{equation*}
\mathcal B_R(x,x^\prime,y,y^\prime) = B(x-y^\prime,R)\cup B(x^\prime-y,R)\cup B(x-y,R).
\end{equation*}
Now, let $\mathcal E_u$ be given by
\begin{align}
\mathcal E_u =& \sup_{l=1,\dots,L} \; \sup_{x,x^\prime,\tilde x,y,y^\prime,\tilde y \in \Qcal} \; \sup_{j \in \Lcal_u} \Eu \Big[
\left|F_l(0,x,x^\prime,y, y^\prime, \bG_{0^c},\bG_{j^c},\bG_{(0j)^c})\right|  \Big\{ \nonumber\\
&R_1(0,x,x^\prime,j,y,y^\prime , \bG_{0^c},\bG_{j^c},\bG_{(0j)^c}) 
\left| A(0,x,x^\prime,j,y,y^\prime , \bG_{0^c},\bG_{j^c},\bG_{(0j)^c})  -1\right| +\nonumber\\
&R_2(0,\tilde x,x^\prime,j,\tilde y,y^\prime, \bG_{0^c},\bG_{j^c},\bG_{(0j)^c}) \times  \left| A(0,\tilde x,x^\prime,j,\tilde y,y^\prime , \bG_{0^c},\bG_{j^c},\bG_{(0j)^c})  -1\right|
\Big\}
\Big]. \label{eq:finiteEu}
\end{align}
Now, we have on the one hand
\begin{align*}
T_1^\prime &\le  \mathcal E_u \int\int\int\int   \sum_{j \in \mathcal B_R(x,x^\prime,y,y^\prime)} %\cap (W_n\ominus (m_n+R))_{-y}} 
\lambda_1(x) \lambda_1(x^\prime)\lambda_1(y) \lambda_1(y^\prime) 
\d x \d x^\prime \d y \d y^\prime \\
&\le  3\mathcal E_u  \int \int  
\Big\{\#(\Lcal_u \cap B(x-y,R)) \Big\}
\lambda_1(x) \lambda_1(y) \d x \d y \\
&\le 3 \mathcal E_u  \left\{ \sup_{x,y\in \Qcal} \#(\Lcal_u \cap B(x-y,R)) \right\} \int \int  \lambda_1(x) \lambda_1(y) \d x \d y= O(1).
\end{align*}
On the other hand,
\begin{align*}
T_2^\prime & \le \mathcal E_u \int\int\int\int\int\int  \\
& \sum_{j \in \mathcal B_R(\tilde x,x^\prime,\tilde y, y^\prime)\cap (W_n\ominus (m_n+R))_{-y}}     \!\!\!\!\!\!\!\!\!\!\!\!\!\!
\lambda_1(x) \lambda_1(x^\prime) \lambda_1(\tilde x)  \lambda_1(y) \lambda_1(y^\prime) \lambda_1(\tilde y)
\d x \d x^\prime \d \tilde x \d y \d y^\prime \d \tilde y\\
&\le 3 \mathcal E_u \left\{ \sup_{x,y\in \Qcal} \#(\Lcal_u \cap B(x-y,R)) \right\} \int \int  \lambda_1(x) \lambda_1(y) \d x \d y= O(1)
\end{align*}
whereby we deduce, recalling~\eqref{eq:asympT1}, that $T_1+T_2 =  \# (\Lcal_u \cap W_n\ominus (m_n+R)) \times (O(1)+ T_1^\prime +T_2^\prime)=O(|W_n|)$, which finally leads to the result.

(ii) We now focus on the consistency of $\hat \theta$. Let $k \in \R^p$, let  $u_n=c v_n^{1/2}|W_n|^{-1}$ for some $c>0$, we can observe that $u_n\to0$ as $n\to \infty$. Let also $\Delta_n(k) = \mathrm{TF}_{n}(\bff,\bG;\theta^\star) - \mathrm{TF}_{n}(\bff,\bG;\theta^\star + u_n k)$. Since $\hat \theta= \mathrm{argmax}_\theta (-\mathrm{TF}_{n}(\bff,\bG;\theta))$, the result is proved if for any $\varepsilon>0$, there exists $K=K(U)$ such that 
\begin{equation*}
    \mathrm P \left( \sup_{|k|=K}\Delta_n(k) > 0  \mid U=u \right) \le \varepsilon .
\end{equation*}
We shorten and introduce a few notation. 
\begin{align*}
e_n(\theta)  &= \mathrm{TF}^{(1)}_n (\bff, \bg; \theta)   = 2 \sum_{l=1}^L \mathrm{DLR}^{(1)}_n(f_l,\bG,\theta) \mathrm{DLR}_n(f_l, \bG,\theta)\\
\check e_n(\theta) &= 2 \sum_{l=1}^L \mathrm{DLR}^{(1)}_n(f_l,\bG,\theta) \widetilde{\mathrm{DLR}}_n(f_l, \bG,\theta)\\
\widetilde{e}_n(\theta) &= 2 |W_n|\sum_{l=1}^L \mathcal I_l(\theta^\star) \widetilde{\mathrm{DLR}}_n(f_l, \bG,\theta)
\end{align*}
Under assumption~\ref{M:f}, there exists $t\in(0,1)$ such that, using a Taylor expansion, we can decompose $\Delta_n(k)$ as
\begin{align*}
\Delta_n(k) =& \;-u_n k^\top \mathrm{TF}^{(1)}_n(\bff,\bG, \theta^\star+ tu_n k) \\
&= \;- u_n k^\top e_n( \theta^\star+ tu_n k) \\
&= T_1+T_2+T_3+T_4
\end{align*}
where, using the short notation $\theta_n^\star = \theta^\star+tu_n k$,
\begin{align*}
T_1&= u_n k^\top \left( \check{e}_n(\theta^\star_n) - e_n(\theta^\star_n) \right) \\
T_2& = u_n k^\top \left( \check{e}_n(\theta^\star) - \check{e}_n(\theta^\star_n)\right)\\
T_3 &=u_n k^\top \left( \widetilde e_n(\theta^\star) - \check{e}_n(\theta^\star)\right) \\
T_4&= -u_n k^\top \widetilde e_n(\theta^\star) 
\end{align*}
We now treat these four terms separately. In the following, we abuse notation by denoting $o_{a.s.}$ the almost sure behaviour with respect to $\P_u$.

\paragraph{Term $T_1$} 
We have that
\begin{align*}
\check{e}_n(\theta^\star_n) - e_n(\theta^\star_n) &= 2 \sum_l  
\mathrm{DLR}^{(1)}_n (f_l,\bG;\theta_n^\star)
\left(
\widetilde{\mathrm{DLR}}_n (f_l,\bG;\theta^\star_n) -\mathrm{DLR}_n (f_l,\bG;\theta_n^\star) 
\right)
\end{align*}
Let $\theta\in \Theta$. On the other hand, we can apply Lemma~\ref{lem:DLRk} and obtain $$\mathrm{DLR}^{(1)}_n (f_l,\bG;\theta)= O_{a.s.}(|W_n|).$$ We may also apply, under assumptions~\ref{M:model}-\ref{M:ergodic}, Theorem~\ref{thm:ergodic}(ii) to $f=f_l$ to derive 
\begin{equation*}
    \widetilde{\mathrm{DLR}}_n (f_l,\bG;\theta^\star_n) -\mathrm{DLR}_n (f_l,\bG;\theta_n^\star) = O_{\P_u}(|W_n|^{1-\beta}).
\end{equation*}
Therefore, for any $\varepsilon>0$, we have, as $n\to \infty$
\begin{equation*}
|W_n|^{-2+\beta-\varepsilon} \left(
\check{e}_n(\theta^\star_n) - e_n(\theta^\star_n)
\right) \stackrel{\P_u}{\to } 0
\end{equation*}
whereby we deduce by the continuous mapping theorem that
\begin{align*}
u_n\left( \check{e}_n(\theta^\star_n) - e_n(\theta^\star_n)\right) &=
O_{\P_u}(u_n|W_n|^{2-\beta+\varepsilon}) \\
&= O_{\P_u}\left( v_n \frac{|W_n|^{1-\beta+\varepsilon}}{\sqrt{v_n}} \right)\\
&= o_{\P_u}(v_n)
\end{align*}
for some small $\varepsilon>0$ and as soon as $\beta>1/2$. Hence, $T_1\le |k|T_1^\prime$ with $T_1^\prime=o_{\P_u}(v_n)$.

%$S_1(\theta) = O_{\P_u}(u_n|W_n|^{2-2/d}\log(|W_n|)^2)$. %Therefore, for any $\varepsilon>0$, $u_n^{-1}|W_n|^{2-2/d}\log(|W_n|)^{2+\varepsilon} S_1(\theta) \stackrel{\P_u}{\to} 0$ as $n\to \infty$ whereby, by continuous mapping theorem, we deduce that 
%\begin{equation*}
%£S_1(\theta_n^\star) = o_{\P_u}( u_n|W_n|^{2-2/d} \log(|W_n|)^{2+\varepsilon}).
%\end{equation*}
%By assumption, 
%\begin{equation}
%    u_n |W_n|^{2-2/d} \log(|W_n|)^{2+\varepsilon} = v_n \frac{|W_n|^{1-2/d} \log(|W_n|)^{2+\varepsilon}}{\sqrt{v_n}} = o(v_n)
%\end{equation}
%Hence, $S_1(\theta_n^\star)  = o_{\P_u}(v_n)$.
%For the term, $S_2$, there exists $s\in(0,1)$ such that by denoting $\check \theta_n^\star= \theta^\star +st u_n k$,
%\begin{equation*}
%k^\top S_2=u_n^2 \sum_l k^\top %\widetilde{\mathrm{DLR}}_n^{(2)} (f_l,\bG;\check %\theta_n^\star) k \;\Delta_n(\theta^\star).
%\end{equation*}
%Now, we combine the continuous mapping theorem and Lemma~\ref{lem:DLRk}, to prove that 
%\begin{equation*}
%    |W_n|^{-1} \widetilde{\mathrm{DLR}}_n^{(2)} (f_l,\bG;\check \theta_n^\star) \stackrel{a.s.}{\to} \mathcal D^{(2)}(f_l, \bG ; \theta^\star).
%\end{equation*}
%Hence, using in addition, Theorem~\ref{thm:ergodic}(ii) to $f=f_l$, we obtain that
%\begin{align*}
%k^\top S_2 &= u_n^2 \; O_{a.s.}(|W_n|) \; O_{\P_u}(|W_n|^{1-1/d} \log|W_n|)  \\
%&=  O_{\P_u}(v_n|W_n|^{-1/d} \log|W_n|) = o_{\P_u}(v_n).
%\end{align*}
%\JF{Fuck Problème sur $S_3$. Le seul truc qu'on ait c'est}\\
%\JF{$S_3=u_n O_{a.s.}(|W_n|)O_\P(|W_n|^{1-1/d}\log|W_n|)$ et c'est mort }

\paragraph{Term $T_2$} Under assumption~\ref{M:f}, there exists $s\in (0,1)$ such that 
\begin{equation*}
T_2 = -t u_n^2 \; k^\top \check e_n^{(1)}(\check\theta^\star_n) k
\end{equation*}
where $\check\theta_n^\star= \theta^\star+st k u_n\to \theta^\star$ as $n\to\infty$, and where for any $\theta \in \Theta$
\begin{align*}
\check e_n^{(1)}(\theta) &= 2 \sum_{l=1}^L \Big\{
\mathrm{DLR}^{(2)}_n(f_l,\bG;\theta)\widetilde{\mathrm{DLR}}_n(f_l,\bG;\theta)+\\
& \qquad \qquad \mathrm{DLR}^{(1)}_n(f_l,\bG;\theta)\widetilde{\mathrm{DLR}}_n^{(1)}(f_l,\bG;\theta)^\top
\Big\}.
\end{align*}
Lemma~\ref{lem:DLRk} and Theorem~\ref{thm:ergodic} show that for any $\theta\in\Theta$, $\P_u$-a.s. as $n\to \infty$
\begin{equation*}
|W_n|^{-2} \check e_n^{(1)}(\theta) \stackrel{a.s.}{\longrightarrow}    2 \sum_{l=1}^L \left\{
\Dcal^{(2)}(f_l;\theta)\Dcal(f_l;\theta) + 
\Dcal^{(1)}(f_l;\theta)\Dcal^{(1)}(f_l;\theta)^\top
\right\}.
\end{equation*}
By the continuous mapping theorem, we deduce that $\P_u$-a.s. as $n\to \infty$
\begin{align*}
|W_n|^{-2} \check e_n^{(1)}(\check\theta^\star_n)
&\stackrel{a.s.}{\longrightarrow} \;    2 \sum_{l=1}^L \left\{
\Dcal^{(2)}(f_l;\theta^\star)\Dcal(f_l;\theta^\star) + 
\Dcal^{(1)}(f_l;\theta^\star)\Dcal^{(1)}(f_l;\theta^\star)^\top
\right\}\\ 
&=:\;  2\sum_{l=1}^L \mathcal I_l(\theta^\star)\mathcal I_l(\theta^\star)^\top = 2 \mathcal I(\theta^\star).
\end{align*}
We finally obtain that
\begin{align*}
T_2 &= -t u_n^2|W_n|^2 |k|^2 \frac{ k^\top (2\mathcal I(\theta^\star))k}{|k|^2} +\\
& \qquad t u_n^2 |W_n|^2 k^\top 
\left\{
2\mathcal I(\theta^\star) - |W_n|^{-2} \check e_n^{(1)}(\check \theta_n^\star) 
\right\} k \\
& \le - \kappa |k|^2  u_n^2|W_n|^2  + |k| T_2^\prime
\end{align*}
where $\kappa = 2t\lambda_{\min}(\mathcal I(\theta^\star)) >0$
and 
\begin{align*}
    T_2^\prime&=\sup_{k, \|k\|=K} |k|u_n^2|W_n|^2  \left|
2\mathcal I(\theta^\star) - |W_n|^{-2} \check e_n^{(1)}(\check\theta^\star_n) 
\right| \\
&= o_{a.s.}(u_n^2 |W_n|^2) = o_{a.s.}(v_n).
\end{align*}

\paragraph{Term $T_3$} We can observe that 
\begin{align*}
\tilde e_n(\theta^\star)- \check e_n(\theta^\star) = 2 |W_n|\sum_{l=1}^L \left\{ 
 |W_n|^{-1}\mathrm{DLR}_{n}^{(1)}(f_l,\bG;\theta^\star) - \mathcal I_l(\theta^\star) 
\right\} \widetilde{\mathrm{DLR}_{n}}(f_l,\bG;\theta^\star).
\end{align*}
Under assumptions~\ref{M:model}, \ref{M:f} and~\ref{M:ergodic}, one can apply Lemma~\ref{lem:DLRk} and Theorem~\ref{thm:convTF}(i) to show that
\begin{equation*}
|W_n|^{-1} \mathrm{DLR}^{(1)}_{n}(f_l,\bG;\theta^\star) - \mathcal I_l(\theta^\star) 
 = o_{a.s.}(1) 
 \quad \text{ and }  \quad
 \widetilde{\mathrm{DLR}}_{n}(f_l,\bG,\theta^\star)= O_{\mathbb P_u}(v_n^{1/2}).
\end{equation*}
 Since $u_n|W_n|=O(v_n^{1/2})$, we deduce that $T_3 \le |k|T_3^\prime$ with $T_3^\prime= o_{\mathbb P_u}(v_n)$.\\

\paragraph{Term $T_4$} By definition and Proposition~\ref{prop:sumDLR}, the random variable $\widetilde e_n(\theta^\star)$ is centered. Therefore, we have in particular from Theorem~\ref{thm:convTF}(i) and conditional Markov's inequality, that $\widetilde e_n(\theta^\star)= O_{\P_u}(\sqrt{\Var_u[\widetilde e_n(\theta^\star)]}) = O_{\P_u}(v_n^{1/2})$ %\C{TO DO à vérifier}.

\paragraph{Conclusion} We have that
\begin{align*}
    \Delta_n(k) \le u_n|W_n| \; |k| |\tilde e_n(\theta^\star)| - \kappa |k|^2 u_n^2|W_n|^2  +|k|T 
\end{align*}
with $T=T_1^\prime+T_2^\prime+T_3^\prime$. Let $u\in \Dcal$, we then have
\begin{align*}
\mathrm P \bigg( \sup_{|k|=K}\Delta_n(k) > 0  \mid U=u \bigg) \le 
\mathrm P \bigg( 
|\tilde e_n(\theta^\star)| + u_n^{-1}|W_n|^{-1} T \ge u_n|W_n| \kappa K  \mid U=u
\bigg).
\end{align*}
Now, the choice of the sequence $u_n$ leads to $u_n|W_n|=c v_n^{1/2}$ and to 
\begin{equation*}
|\tilde e_n(\theta^\star)| +   u_n^{-1}|W_n|^{-1} T= O_{\P_u}(v_n^{1/2}) + v_n^{-1/2} o_{\P_u}(v_n) = O_{\P_u}(v_n^{1/2})
\end{equation*}
whereby we deduce the result.
\end{proof}
\end{appendix}

\begin{acks}[Acknowledgments]
The  research  of C. Renaud Chan and J.-F.  Coeurjolly is  supported  by Persyval-lab (ANR-11-61 LABX-0025-01). The authors would like to thank David Dereudre for insightful discussions and Rémy Drouilhet for assistance with the implementation of the simulation algorithm.
\end{acks}

\bibliographystyle{imsart-number} 
\bibliography{refs}

\end{document}